\documentclass[12pt]{article}
\usepackage[latin1]{inputenc}
\usepackage{amsmath}
\usepackage{amsfonts}
\usepackage{amssymb}
\usepackage{amsthm}

\usepackage{titlesec}
\titleformat{\subsection}[hang]{\normalfont\bfseries}{\thesubsection}{1em}{}
\titleformat{\section}[hang]{\normalfont\large\bfseries}{\thesection}{1em}{}

\usepackage{titlesec}

\titlespacing\section{0pt}{3.5ex plus 0.5ex minus .2ex}{0.3ex plus .2ex}
\titlespacing\subsection{0pt}{2.5ex plus 0.5ex minus .2ex}{0.3ex plus .2ex}
\titlespacing\subsubsection{0pt}{2.5ex plus 0.5ex minus .2ex}{0.3ex plus .2ex}

\usepackage{mathrsfs} 
\usepackage{graphicx}

\usepackage[alphabetic,lite]{amsrefs} 
\usepackage{mathdots} 

\usepackage{fullpage}
\usepackage{setspace}
\usepackage{hyperref}
\usepackage{color}
\usepackage{enumerate} 
\usepackage{ulem} 
\usepackage{comment}

\usepackage[all,cmtip]{xy} 

 \usepackage{fancyhdr}
\numberwithin{equation}{subsection}
\newtheorem{Thm}{Theorem}[subsection]

\newtheorem{Lemma}[Thm]{Lemma}

\newtheorem{Prop}[Thm]{Proposition}
\newtheorem{Fact}[Thm]{Fact}

\theoremstyle{definition}
\newtheorem{Def}[Thm]{Definition}
\newtheorem{Rem}[Thm]{Remark}
\newtheorem{Aside}[Thm]{Aside}

\newtheorem{Not}[Thm]{Notation}

\newcommand{\<}{\left\langle}
\renewcommand{\>}{\right\rangle}
\newcommand{\ceil}[1]{\left\lceil #1 \right\rceil}  

\newcommand{\bC}{\mathbb{C}}

\newcommand{\bF}{\mathbb{F}}
\newcommand{\bG}{\mathbb{G}}

\newcommand{\bQ}{\mathbb{Q}}
\newcommand{\bR}{\mathbb{R}}

\newcommand{\bZ}{\mathbb{Z}}

\newcommand{\cO}{\mathcal{O}}

\newcommand{\fg}{\mathfrak{g}}

\newcommand{\fr}{\mathfrak{r}}
\newcommand{\fs}{\mathfrak{s}}
\newcommand{\ft}{\mathfrak{t}}

\newcommand{\sA}{\mathscr{A}}
\newcommand{\sB}{\mathscr{B}}

\newcommand{\sF}{\mathscr{F}}

\DeclareMathAlphabet{\mathpzc}{OT1}{pzc}{m}{it}

\newcommand{\ra}{\rightarrow}

\newcommand{\wt}{\widetilde}

\newcommand{\eps}{\epsilon}

\providecommand{\abs}[1]{\left\lvert#1\right\rvert}

\DeclareMathOperator{\Image}{Im}		
\DeclareMathOperator{\Hom}{Hom}			
\DeclareMathOperator{\Id}{Id}		  	
\DeclareMathOperator{\GL}{GL}		  	
\DeclareMathOperator{\GSp}{GSp}		  	%
\DeclareMathOperator{\U}{U}		  	
\DeclareMathOperator{\SL}{SL}		  	
\DeclareMathOperator{\Gal}{Gal}			
\DeclareMathOperator{\Ind}{Ind}			
\DeclareMathOperator{\Sp}{Sp}		  	
\DeclareMathOperator{\Ad}{Ad}	  		
\DeclareMathOperator{\Lie}{Lie}  	  
\DeclareMathOperator{\cind}{c-ind}  	  
\DeclareMathOperator{\val}{val}  	  
\DeclareMathOperator{\Cent}{Cent}  	  
\DeclareMathOperator{\der}{der}		

\usepackage{multind}
\makeindex{notation}
\makeindex{definition}

\AtEndDocument{\bigskip{\footnotesize%
		\textsc{Universität Bonn,
			Mathematisches Institut,
			Endenicher Allee 60,
			53115 Bonn,
			Germany } \par  
		\textit{E-mail address}: \texttt{fintzen@math.uni-bonn.de} 
}}

\begin{document}
\author{Jessica Fintzen}
\title{Supercuspidal representations: construction, classification, and characters}
\date{}
\maketitle
\begin{abstract}
The building blocks for irreducible smooth representations of $p$-adic groups are the supercuspidal representations. In these notes that are an expansion of a lecture series given during the IHES summer school 2022 we will explore an explicit exhaustive construction of these supercuspidal representations and their character formulas and observe a striking parallel between a large class of these representations and discrete series representations of real algebraic Lie groups. A key ingredient for the construction of supercuspidal representations is the Bruhat--Tits theory and Moy--Prasad filtration, which we will recall in this survey.
\end{abstract}

{
	\renewcommand{\thefootnote}{}  
	\footnotetext{The author was partially supported by NSF Grants DMS-2055230 and DMS-2044643, a Royal Society University Research Fellowship, a Sloan Research Fellowship and the European Research Council (ERC) under the European Union's Horizon 2020 research and innovation programme (grant agreement n° 950326).}
}

\tableofcontents


\vspace*{1cm}

\textbf{Disclaimer.} There is some overlap between these notes and the ones written by the same author for the Current Developments in Mathematics 2021 conference that took place in March 2022 (\cite{Fi-CDM}). More precisely, Section \ref{Section-MP-BT} below is an expanded version of Section 3 of \cite{Fi-CDM} and some subsections of Sections \ref{Section-construction} and \ref{Section-classification} are based on subsections of Sections 4 and 5 of \cite{Fi-CDM}.
\\

\textbf{Acknowledgments.} 
I thank Tasho Kaletha for feedback on a previous draft and Loren Spice for discussions on Harish-Chandra characters including supplying me with references for the results about characters in positive characteristic. I also thank both of them for discussions related to the topics of this survey.  This article is an expansion of a lecture series given during the IHES summer school 2022 and I thank the organizers of the summer school for inviting me to the school and to contribute an article to the proceedings of the summer school. While writing the final parts of these notes I also benefited from the hospitality of the Max Planck Institute for Mathematics in Bonn. I am grateful to the referee for a careful reading of the manuscript.



\section{Motivation and connection between representations of real and $p$-adic reductive groups}\label{Sect-motivation-and-real}
According to Harish-Chandra's Lefschetz principle, the representation theory of real reductive groups and $p$-adic reductive groups should show striking parallels. 

We therefore start by recalling a few results about representations of real reductive groups and use this to motivate the results we expect for $p$-adic reductive groups.

Throughout the survey we let $F$\index{notation}{F@$F$} be a local field and $G$\index{notation}{G@$G$} a connected reductive group over $F$. 
We denote by lowercase fraktur letters the Lie algebras of reductive groups, e.g., $\fg$ denotes the Lie algebra of $G$.\index{notation}{g@$\fg$}
Starting from Section \ref{Section-MP-BT} we assume that $F$ is non-archimedean.

\subsection{Representations of real algebraic Lie groups}
In this subsection we consider the case of $F=\bR$, i.e.,we consider a connected reductive group $G$ over $\bR$ and study its complex representations. More precisely, we consider homomorphisms from $G(\bR)$ into the group of invertible bounded operators on a separable Hilbert space $V$ such that the resulting map $G(\bR) \times V \ra V$ is continuous.  
By the Langlands classification (\cite{Langlands89}, compare \cite[\S~3.3]{Taibi-IHES}), every such irreducible representation is equivalent to one that is the unique quotient of a parabolic induction from a tempered unitary representation twisted by a character of the center of a Levi subgroup. The tempered unitary representations themselves all occur as irreducible subquotients of essentially square-integrable representations.  

\begin{Rem} Let $K$ be a maximal compact subgroup of $G(\bR)$ (which is unique up to conjugation) and let $\fg_\bC$ be the complexification of the (real) Lie algebra of $G(\bR)$. 
	Irreducible unitary representations are unitarily equivalent if and only if their (irreducible unitary) $(\fg_\bC, K)$-modules are isomorphic, and every irreducible unitary $(\fg_\bC, K)$-module is the $(\fg_\bC, K)$-module of an irreducible unitary representation of $G(\bR)$ (\cite{HC-reps1
		, HC-reps2, HC-Plancherel}, see also \cite[\S2]{Wallach-Corvallis}). Thus parametrizing equivalence classes of irreducible unitary representations is equivalent to parametrizing isomorphism classes of irreducible unitary $(\fg_\bC, K)$-module, which is the notion used in \cite{Taibi-IHES}.
\end{Rem} 

In 1965/66 Harish-Chandra (\cite{HC-DS1, HC-DS2}) provided a classification of essentially square-integrable representations (initially for semisimple groups). In order to state the result, we introduce some notation following \cite[\S~4.11]{Kaletha-regular}. Suppose that $G$ has an elliptic maximal torus $S$ and let $\theta$ be a character of $S(\bR)$. We write $\Phi(G,S)=\Phi(G_\bC,S_\bC)$ for the (absolute) root system of $G$ with respect to $S$ and for $\alpha \in \Phi(G,S)$, we write $\check \alpha \in \Hom(\bG_m, S_\bC)$ for the corresponding coroot. We denote by $G_{\der}$ the derived group of $G$, by $S_{\der}$ the intersection of $S$ with $G_{\der}$ and by $\theta_{\der}$ the restriction of $\theta$ to $S_{\der}(\bR)$. Then $S_{\der}$ is anisotropic, and hence $\theta_{\der}$ is the restriction of an algebraic character of $(S_{\der})_\bC$, and we can identify $\theta_{\der}$ with an element of $\Hom((S_{\der})_\bC, \bG_m)$. Using the usual  $\bZ$-valued pairing $\langle \cdot, \cdot \rangle$  between the $\bZ$-modules $\Hom(\bG_m, (S_{\der})_\bC)$ and $\Hom((S_{\der})_\bC, \bG_m)$ and the observation that coroots factor through $S_{\der}$, we define $\langle \theta, \check \alpha \rangle$ to be $\langle \theta_{\der}, \check \alpha \rangle$ for $\alpha \in \Phi(G,S)$.

This allows us to state Harish-Chandra's classification, extended to general reductive groups, following \cite[\S~4.11]{Kaletha-regular}, which is based on \cite{Langlands89}.

\begin{Thm}
	The equivalence classes of essentially square-integrable representations of $G(\bR)$ are parameterized by triples $(S, \theta, \Phi^+)$ up to conjugacy, where 
	\begin{enumerate}[(i)]
		\item $S$ is an elliptic maximal torus of $G$
		\item $\theta: S(\bR) \ra \bC^\times$ is a character
		\item $\Phi^+$ is a choice of positive roots for $\Phi(G,S)$ such that $\langle \theta, \check \alpha \rangle \geq 0$ for all $\alpha \in \Phi^+$. 
	\end{enumerate}
Moreover, the representations in the equivalence class attached to the conjugacy class of $(S, \theta, \Phi^+)$ are determined by their character $\Theta_{(S, \theta, \Phi^+)}$ satisfying
\begin{equation}
\label{eqn:char:real}
\Theta_{(S, \theta, \Phi^+)}(\gamma)=(-1)^{\frac{1}{2}\dim(G(\bR)/K)}\sum_{w \in N_G(S)(\bR)/S(\bR)}\frac{\theta(w\gamma w^{-1})}{\prod_{\alpha \in \Phi^+}(1-\alpha(w\gamma w^{-1})^{-1})}
\end{equation}
for every regular element $\gamma$ of $S(\bR)$, where $N_G(S)$ denotes the normalizer of $S$ in $G$.

\end{Thm}
In particular, a group $G(\bR)$ has essentially square-integrable representations (also called discrete series) if and only if $G$ has an elliptic maximal torus. Moreover, if $\langle \theta, \check \alpha \rangle \neq 0$ for all $\alpha \in \Phi(G,S)$, then $\Phi^+$ is uniquely determined by $(S, \theta)$. In other words, if $\theta$ is sufficiently regular, i.e., $\langle \theta, \check \alpha \rangle \neq 0$, there is a unique essentially square-integrable representation (up to equivalence) attached to the $G(\bR)$-conjugacy class of pair $(S, \theta)$ and its character is provided by the formula \eqref{eqn:char:real}. In line with the Harish-Chandra Lefschetz principle, we will observe an analogous result in the $p$-adic world.

\subsection{Connection between representations of real and $p$-adic groups} \label{Section-connection-real-p-adic}
 If $F$ is a non-archimedean local field, $p$ is sufficiently large and $G$ is a connected reductive group over $F$ that splits over a tamely ramified extension of $F$, then there exists a notion of regular tame elliptic pairs $(S, \theta)$, see Definition \ref{Def-tame-regular-pair}, consisting of an elliptic maximal torus $S$ of $G$ and a regular character $\theta:S(F) \ra \bC^\times$, to which one can attach a supercuspidal representation $\pi_{(S,\theta)}$ whose Harish-Chandra character on sufficiently regular elements of $S(F)$ is given by \eqref{eqn:char:real} when interpreted appropriately, see Section \ref{Section-character-shallow}, in particular Theorem \ref{Thm-character-formula-shallow}.
 The resulting supercuspidal representations,  introduced and called regular supercuspidal representations by Kaletha (\cite{Kaletha-regular}) and discussed in Section \ref{Section-regular-supercuspidal}, form a large part of the supercuspidal representations as a special case of the general construction of supercuspidal representations in Section \ref{Section-Yus-construction} provided by Yu (\cite{Yu}).

In the setting of real groups, Harish-Chandra gave a classification of all essentially square-integrable representations by providing their characters. An explicit construction of these representations  was achieved about ten years later by Schmid (\cite{Schmid76}) using $L^2$-cohomology as conjectured by Kostant (\cite{Kostant-Conj}) and Langlands (\cite{Langlands-Conj}). In the $p$-adic world, the developments have been the other way round. First mathematicians constructed supercuspidal representations, see Section \ref{Section-history-construction}, and much later their characters were computed, see Section \ref{Section-characters}. It is a folklore conjecture that all supercuspidal representations of $p$-adic groups arise via compact induction from compact-mod-center open subgroups, and this is the case for all known constructions so far. We therefore begin by introducing Bruhat--Tits theory and the Moy--Prasad subgroups in Section \ref{Section-MP-BT}, which provides a framework for studying compact open subgroups of general $p$-adic groups. 


\section{Moy--Prasad filtration and Bruhat--Tits theory} \label{Section-MP-BT}
From now on, throughout the rest of the paper, unless mentioned otherwise, $F$ is a non-archimedean local field and $G$ a connected reductive group over $F$.

The Moy--Prasad filtration is a decreasing filtration of $G(F)$ by compact open subgroups that are normal inside each other and whose intersection is trivial. It is a refinement and generalization of the congruence filtration of $\GL_n(F)$. One usually starts with the definition of a Bruhat--Tits building that Bruhat and Tits (\cite{BT1, BT2}) attached to the reductive group $G$ over $F$ in 1972/1984, and then to each point in the Bruhat--Tits building, Moy and Prasad (\cite{MP1, MP2}) associated in 1994/1996 a filtration by compact open subgroups. In this survey, we will take a different approach and first introduce the Moy--Prasad filtration and use it to define the Bruhat--Tits building. This section is an expanded version of Section 3 of \cite{Fi-CDM}.

 We first introduce some notation that we use throughout the remainder of the survey. For every finite extension $E$ of $F$, we denote the ring of integers of $E$ by $\cO_E$ and a uniformizer by $\varpi_E$.\index{notation}{O@$\cO$}\index{notation}{OE@$\cO_E$}\index{notation}{OFbar@$\cO_{\bar F}$}\index{notation}{Fbar@$\bar F$}
 We might drop the index $E$ if $E=F$ and denote the residue field of $F$ by $\bF_q$\index{notation}{Fq@$\bF_q$}. 
  We also fix a valuation $\val:F \twoheadrightarrow \bZ \cup \{\infty\}$\index{notation}{val}. 

\subsection{The split case}
We assume in this subsection that $G$ is split over $F$. 
Let $T$ be a split maximal torus of $G$ and denote by $\Phi(G, T)$ the root system of $G$ with respect to $T$. We recall that a \textit{Chevalley system} \index{definition}{Chevalley system} $\{X_\alpha\}_{\alpha \in \Phi(G, T)}$ consists of a non-trivial element $X_\alpha$ in the one dimensional $F$-vector root subspace $\fg_\alpha(F) \subset \fg(F)$ for each root $\alpha$ of $G$ with respect to $T$ such that
$$ \Ad(w_\beta)(X_\alpha)=\pm X_{s_\beta(\alpha)} \, ,  \, \forall \alpha, \beta \in \Phi(G,T),$$
where $w_\beta$ is an element of the normalizer $N_G(T)(F)$ of $T$ in $G$ determined by $X_\beta$ whose image in the Weyl group $(N_G(T)/T)(F)$ is the simple reflection $s_\beta$ corresponding to $\beta$. For example,  if $G=\SL_2$ and $X_\beta=\begin{pmatrix}
0 & 1 \\ 
0 & 0
\end{pmatrix}$, then $w_\beta=\begin{pmatrix}
0 & 1 \\ 
-1 & 0
\end{pmatrix}$.
In general $w_\beta$ is defined as follows. 
For every root $\beta$, we let $x_\beta:\bG_a \xrightarrow{\simeq} U_\beta$ be the isomorphism between the additive group $\bG_a$ and the root group $ U_\beta \subset G$ attached to $\beta$  whose derivative sends $1 \in F=\bG_a(F)$ to $X_\beta$. Then 
$$w_\beta=x_\beta(1)x_{-\beta}(\epsilon)x_\beta(1)$$
where $\epsilon\in\{\pm 1\}$ is the unique element for which $x_\beta(1)x_{-\beta}(\epsilon)x_\beta(1)$ lies in the normalizer of $T$.

For example, for $\GL_n$ the collection $\{X_{\alpha_{i,j}}\}_{1 \leq i, j \leq n; i \neq j}$ consisting of the matrices with all entries zero except for a one at position $(i, j)$ forms a Chevalley system. 

This allows us to make the following definition, but we warn the reader that we have not seen anyone else use the terminology ``BT triple''.
\begin{Not}
	A \textit{BT triple} $(T, {X_\alpha}, x_{BT})$ consists of \index{definition}{BT triple}
	\begin{enumerate}[(i)]
		\item a split maximal torus $T$ of $G$,
		\item a Chevalley system $\{X_\alpha\}_{\alpha \in \Phi(G, T)}$, and
		\item $x_{BT} \in X_*(T)\otimes_{\bZ} \bR:= \Hom_F(\bG_m, T)\otimes_{\bZ} \bR$.
	\end{enumerate}
\end{Not} \index{notation}{X*Tupper@$X^*(T)$} \index{notation}{X*Tlower@$X_*(T)$}
Here $\bG_m$ denotes the multiplicative group scheme and $\Hom_F$ denotes homomorphisms in the category of $F$-group schemes. Then $X_*(T):=\Hom_F(\bG_m, T)$ is a free $\bZ$-module, hence $X_*(T)\otimes_{\bZ} \bR$ is a finite-dimensional real vector space. Moreover, we have a bilinear pairing between  $X^*(T):=\Hom_F(T,\bG_m)$ and $X_*(T)= \Hom_F(\bG_m, T)$ obtained by identifying $\Hom_F(\bG_m, \bG_m)$ with $\bZ$. We extend this map $\bR$-linearly in the second factor to obtain a map 
$$\< \cdot , \cdot \>:  X^*(T) \times   X_*(T)\otimes_{\bZ} \bR \ra \bR . $$
In particular, we may pair $x_{BT}$ with a root $\alpha \in \Phi(G,T)$ to obtain a real number $\<\alpha, x_{BT} \>$.

We now fix a BT triple $x=(T, \{X_\alpha\}, x_{BT})$ and define the Moy--Prasad filtration attached to it.

\textbf{Filtration of the torus.}

We set 
$$T(F)_0 = \{ t \in T(F) \, | \, \val(\chi(t))=0 \,  \, \forall \, \chi \in X^*(T)=\Hom_F(T,\bG_m)\} ,$$
which is the maximal bounded subgroup of $T(F)$. For $r \in \bR_{>0}$, we define
$$T(F)_r = \{ t \in T(F)_0 \, | \, \val(\chi(t)-1) \geq r \, \, \forall \,  \chi \in X^*(T)\} .$$
For example, if $G=\GL_n$ and $T$ is the torus consisting of diagonal matrices, then $T(F)_0$ consists of diagonal matrices whose entries are all in the invertible element $\cO^\times$ of $\cO$ and $T(F)_r$ consists of diagonal matrices whose entries are all in $1 + \varpi^{\ceil{r}}\cO$, where we recall that $\varpi$ is a uniformizer of $F$.

\textbf{Filtration of the root groups.}

Let $\alpha \in \Phi(G,T)$. We recall that the isomorphism $x_\alpha:\bG_a \ra U_\alpha$ is defined by requiring its derivative $dx_\alpha$ to send $1\in F=\bG_a(F)$ to $X_\alpha$.
For $r \in \bR_{\geq 0}$, we define the filtration subgroups of $U_\alpha(F)$ as follows\index{notation}{UalphaF@$U_{\alpha}(F)_{x,r}$}
$$ U_{\alpha}(F)_{x,r}:= x_\alpha(\varpi^{\ceil{r-\<\alpha, x_{BT}\>}}\cO) . $$
Let us consider the example of $G=\SL_2$ and $T$ the torus consisting of diagonal matrices.
\textbf{Example 1.} Let $x_1$ be the Bruhat--Tits triple 
$(T,
\left\{\begin{pmatrix}
0 & 1 \\ 0 & 0
\end{pmatrix},
\begin{pmatrix}
0 & 0 \\ 1 & 0
\end{pmatrix}
\right\},
0)$.
Let $\alpha$ correspond to the map $\begin{pmatrix}
t & 0 \\ 0 & t^{-1}
\end{pmatrix}
\mapsto t^2$  for $t \in F^\times$. Then $x_\alpha(y)=\begin{pmatrix}
1 & y \\ 0 & 1
\end{pmatrix}$ for $y \in F=\bG_a(F)$ and
$$ U_{\alpha}(F)_{x_1, r}=
\begin{pmatrix}
1 & \varpi^{\ceil{r}}\cO \\ 0 & 1
\end{pmatrix}
\, \, \text{ and } \, \,
U_{-\alpha}(F)_{x_1, r}=
\begin{pmatrix}
1 & 0 \\  \varpi^{\ceil{r}}\cO & 1
\end{pmatrix}.
$$
\textbf{Example 2.} \label{Explx2} Let $x_2$ be the Bruhat--Tits triple 
$(T,
\left\{\begin{pmatrix}
0 & 1 \\ 0 & 0
\end{pmatrix},
\begin{pmatrix}
0 & 0 \\ 1 & 0
\end{pmatrix}
\right\},
\frac{1}{4}\check \alpha)$, where $\check \alpha$ is the coroot of $\alpha$, i.e., the element of $X_*(T)$ that satisfies $\check\alpha(t)=\begin{pmatrix}
t & 0 \\
0 & t^{-1}
\end{pmatrix}$ for $t \in F^\times=\bG_m(F)$. 
Then
$$ U_{\alpha}(F)_{x_2, r}=
\begin{pmatrix}
1 & \varpi^{\ceil{r-\frac{1}{2}}}\cO \\ 0 & 1
\end{pmatrix}
\, \, \text{ and } \, \,
U_{-\alpha}(F)_{x_2, r}=
\begin{pmatrix}
1 & 0 \\  \varpi^{\ceil{r+\frac{1}{2}}}\cO & 1
\end{pmatrix}.
$$

\textbf{Filtration of $G(F)$.}\\
We define the filtration subgroup $G(F)_{x,r}$ of $G(F)$ for $r \in \bR_{\geq 0}$ to be the subgroup generated by $T(F)_r$ and $U_{\alpha}(F)_{x,r}$ for all roots $\alpha$, i.e.
$$ G(F)_{x,r}=\<T(F)_r, U_{\alpha}(F)_{x,r} \, | \, \alpha \in \Phi(G,T)\> .$$
If the ground field $F$ is clear from the context, we may also abbreviate $G(F)_{x,r}$ by $G_{x,r}$.
\index{notation}{Gxr@$G_{x,r}$} \index{notation}{GFxr@$G(F)_{x,r}$}

In the example of  $G=\SL_2$ for the two Bruhat--Tits triples above, we have for $r>0$
\[G_{x_1,0}=\SL_2(\cO) 
\, , \, 
\, \,
G_{x_1,r}=
\begin{pmatrix}
1+\varpi^{\ceil{r}}\cO & \varpi^{\ceil{r}}\cO \\  \varpi^{\ceil{r}}\cO & 1+\varpi^{\ceil{r}}\cO
\end{pmatrix}_{\det=1}
\]
and 
\[G_{x_2,0}=
\begin{pmatrix}
\cO & \cO \\  \varpi\cO & \cO
\end{pmatrix}_{\det=1}
\, , \, 
\, \,
G_{x_2,r}=
\begin{pmatrix}
1+\varpi^{\ceil{r}}\cO & \varpi^{\ceil{r-\frac{1}{2}}}\cO \\  \varpi^{\ceil{r+\frac{1}{2}}}\cO & 1+\varpi^{\ceil{r}}\cO
\end{pmatrix}_{\det=1} .
\]

\textbf{Filtration of $\fg(F)$ and $\fg^*(F)$ .}\\
One can analogously define a filtration $\fg_{x,r}$ of the Lie algebra $\fg(F)$ and a filtration $\fg^*_{x,r}$ of the $F$-linear dual $\fg^*(F)$ of the Lie algebra $\fg(F)$ as follows. Let $r$ be a real number, and recall that we write $\ft$ for the Lie algebra of the torus $T$.
Then we set 
$$\ft(F)_r = \{ X \in \ft(F) \, | \, \val(d\chi(X)) \geq r \, \, \forall \,  \chi \in X^*(T)\} ,$$
where $d \chi$ denotes the derivative of $\chi$, 
$$ \fg_{\alpha}(F)_{x,r}= \varpi^{\ceil{r-\<\alpha, x_{BT}\>}}\cO \, X_\alpha \subset \fg_\alpha(F) $$
for $\alpha \in \Phi(G,T)$, and 
$$ \fg(F)_{x,r}=\ft(F)_r \oplus \bigoplus_{\alpha \in \Phi(G,T)} \fg_{\alpha}(F)_{x,r}. $$
We define the filtration subspace $\fg^*(F)_{x,r}$ of the dual of the Lie algebra by
$$ \fg^*(F)_{x,r}=\{ X \in \fg^*(F) \, | \, X(Y) \in \varpi \cO \, \text{ for all } \, Y \in \fg(F)_{x,s} \, \text{ with } s>-r \} . $$
If the ground field $F$ is clear from the context, we may also abbreviate $\fg(F)_{x,r}$ and $\fg^*(F)_{x,r}$ by $\fg_{x,r}$ and $\fg^*_{x,r}$, respectively.
\index{notation}{gxr@$\fg_{x,r}$} \index{notation}{gFxr@$\fg(F)_{x,r}$} \index{notation}{gxr*@$\fg^*_{x,r}$} \index{notation}{gFxr*@$\fg^*(F)_{x,r}$}

\subsubsection{Properties of the Moy--Prasad filtration}

\begin{Def}
	A \textit{parahoric} subgroup of $G$ is a subgroup of the form $G_{x,0}$ for some BT triple $x$.
\end{Def}
\index{definition}{parahoric subgroup}

For $r \in \bR_{\geq 0}$,  we write $G_{x,r+}=\bigcup_{s>r}G_{x,s}$ and $\fg_{x,r+}=\bigcup_{s>r}\fg_{x,s}$.\index{notation}{Gxr+@$G_{x,r+}$}\index{notation}{gxr+@$\fg_{x,r+}$}

We collect a few facts about this filtration that demonstrate the richness of its structure.
\begin{Fact} \label{Fact-BT-triple-properties} Let $x$ be a BT triple.
	\begin{enumerate}[(a)] 
	\item $G_{x,r}$ is a normal subgroup of $G_{x,0}$ for all $r \in \bR_{\geq 0}$.
	\item The quotient $G_{x,0}/G_{x,0+}$ is the group of the $\bF_q$-points of a reductive group $\bG_x$ defined over the residue field $\bF_q$ of $F$.
	\item For $r \in \bR_{> 0}$, the quotient $G_{x,r}/G_{x,r+}$ is abelian and can be identified with an $\bF_q$-vector space.
	\item \label{Facts-filtrations-reps} Let $r>0$. Since $G_{x,r}$ is a normal subgroup of $G_{x,0}$, the group $G_{x,0}$ acts on $G_{x,r}$ via conjugation. This action descends to an action of the quotient $G_{x,0}/G_{x,0+}$ on the vector space $G_{x,r}/G_{x,r+}$ and the resulting action is (the $\bF_q$-points of) a linear algebraic action, i.e., corresponds to a morphism from $\bG_x$ to $\GL_{\dim_{\bF_q}(G_{x,r}/G_{x,r+})}$ over $\bF_q$.
	\item We have the following isomorphism that is often referred to as the ``Moy--Prasad isomorphism'': $G_{x,r}/G_{x,r+} \simeq \fg_{x,r}/\fg_{x,r+}$ for $r \in \bR_{> 0}$ and more general $G_{x,r}/G_{x,s} \simeq \fg_{x,r}/\fg_{x,s}$ for $r, s \in \bR_{> 0}$ with $2r \geq s>r$.
\end{enumerate}
\end{Fact}
		
In fact we have a rather good understanding of the representations occurring in (\ref{Facts-filtrations-reps}). In \cite{Fi-MP} they are described explicitly in terms of Weyl modules. Previously they were also realized using Vinberg--Levy theory by Reeder and Yu (\cite{ReederYu}), which was generalized in \cite{Fi-MP}.	
		
\subsubsection{The Bruhat--Tits building}	
\begin{Def}
	The \textit{(reduced) Bruhat--Tits building} $\sB(G, F)$\index{notation}{BGF@$\sB(G, F)$}\index{definition}{Bruhat--Tits building} of $G$ over $F$ is as a set the quotient of the set of BT triples by the following equivalence relation: Two BT triple $x_1$ and $x_2$ are equivalent if and only if $G_{x_1, r}=G_{x_2,r}$ for all $r \in \bR_{\geq 0}$.
\end{Def}
As a consequence of the definition, for $x \in \sB(G,T)$, we may write $G_{x,r}$ for the Moy--Prasad filtration attached to any BT triple in the equivalence class of $x$.

The Bruhat--Tits building $\sB(G,F)$ admits an action of $G(F)$ that is determined by the property 
$$G_{g.x,r}=gG_{x,r}g^{-1} \, \, \forall \, r \in \bR_{\geq 0}, g \in G(F) .$$

We will now equip the Bruhat--Tits building with more structure. 

\textbf{Apartments as affine spaces.} \index{notation}{ATF@$\sA(T,F)$}
\begin{Def}
	For a split maximal torus $T$, we call the subset of $\sB(G,F)$ that can be represented by BT triples whose first entry is the given torus $T$, i.e.
	$$ \sA(T,F):=\{(T,\{X_\alpha\}, x_{BT})\}/_\sim \, \, \subset \sB(G,F)$$
the \textit{apartment} of $T$. \index{definition}{apartment}
\end{Def}
We fix a split maximal torus $T$ and a Chevalley system $\{X_\alpha\}_{\alpha \in \Phi(G,T)}$. Then it turns out that every element in $\sA(T,F)$ can be represented by a BT triple whose first two entries are the torus $T$ and the fixed Chevalley system $\{X_\alpha\}_{\alpha \in \Phi(G,T)}$. Moreover, two BT triples $(T,\{X_\alpha\}, x_{BT, 1})$ and $(T,\{X_\alpha\}, x_{BT, 2})$ are equivalent if and only if $x_{BT, 2}-x_{BT, 1}$ lies in the subspace $X_*(Z(G))\otimes \bR$, where $Z(G)$\index{Z(G)@$Z(G)$}  denotes the center of $G$. Note that $X_*(Z(G))\otimes \bR$ is trivial when the center $Z(G)$ of $G$ is finite. Thus the set $\sA(T,F)$ is isomorphic to $X_*(T)\otimes \bR / X_*(Z(G))\otimes \bR$, and we use this isomorphism to equip $\sA(T,F)$ with the structure of an affine space over the real vector space $X_*(T)\otimes \bR / X_*(Z(G))\otimes \bR$. While the isomorphism of $\sA(T,F)$ with $X_*(T)\otimes \bR / X_*(Z(G))\otimes \bR$ depends on the choice of the  Chevalley system $\{X_\alpha\}_{\alpha \in \Phi(G,T)}$, the structure of $\sA(T,F)$ as an affine space does not. In fact, the choice of a Chevalley system turns the affine space into a vector space by choosing a base point.

\textbf{Polysimplicial structure on apartments.}

Let $T$ be a split maximal torus of $G$. For $\alpha \in \Phi(G,T)$, we define the following set of hyperplanes of the apartment $\sA(T, F)$:
$$ \Psi_\alpha:=\left\{\text{ hyperplanes } H \subset \sA(T,F) \text{ satisfying } \begin{array}{ll}
U_{\alpha}(F)_{x,0}=U_{\alpha}(F)_{y,0} & \forall x, y \in H \\
U_{\alpha}(F)_{x,0} \neq U_{\alpha}(F)_{x,0+} & \forall x \in H 
\end{array} \right\}  . $$
We set $$ \Psi := \bigcup_{\alpha \in \Phi(G,T)} \Psi_\alpha $$ and use these hyperplanes to turn the apartment $\sA(T,F)$ into the geometric realization of a polysimplicial complex. This means the connected components of the complement of the union of the hyperplanes in $\Psi$ are the maximal dimensional polysimplices, which are also called \textit{chambers}. \index{definition}{chamber}

\begin{figure}[h]
	\centering
	\includegraphics[width=0.700\textwidth]{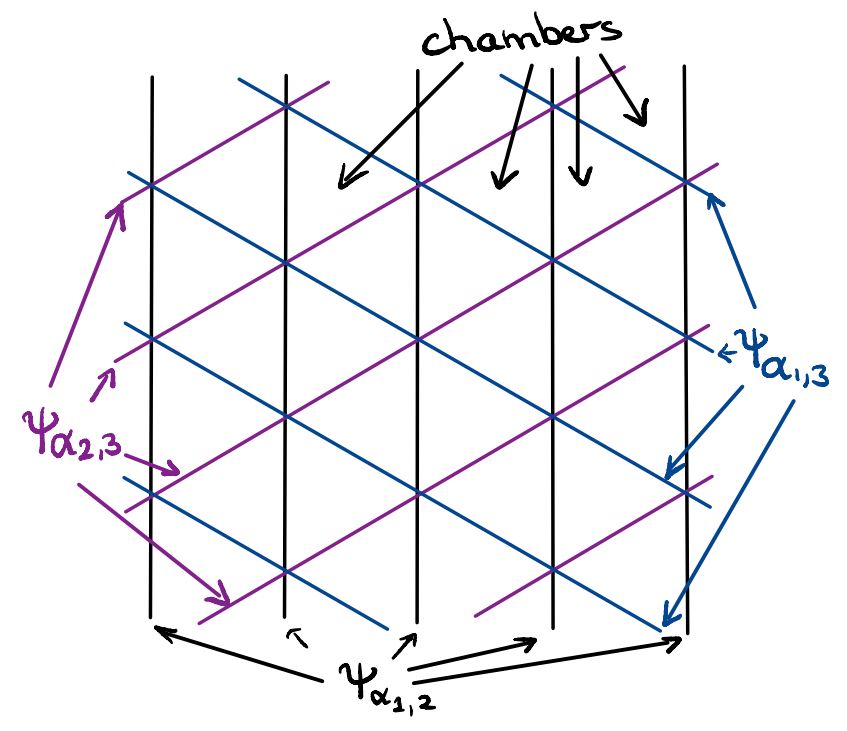}
	\caption{Excerpt of an apartment for $\SL_3$ with hyperplanes, where $\alpha_{i, j}$ is the root corresponding to $\text{diag}(t_1, t_2, t_3) \mapsto t_it_j^{-1}$}
	\label{figure-hyperplanes-sl3}
\end{figure} 

\textbf{Polysimplicial structure on the Bruhat--Tits building.}
The polysimplicial structure on the apartments yields a polysimplicial structure on the Bruhat--Tits building $\sB(G,F)$, which satisfies the properties of an abstract building. In order to recall the notion of an abstract building, we need to introduce some notation following \cite[\S1.5]{KP-BTbook}. 

A \textit{chamber complex} is a polysimplicial complex $\sB$ in which every facet is contained in a maximal facet, called \textit{chamber}, and given two chambers $C$ and $C'$ there exists a sequence $C=C_1\neq C_2 \neq C_3 \neq \hdots \neq C_n = C'$ such that $C_i \in \sB, C_i \cap C_{i+1} \in \sB$ and $\not\exists C''$ with $C_i \cap C_{i+1} \subsetneq C'' \subsetneq C_i$ or $C_i \cap C_{i+1} \subsetneq C'' \subsetneq C_{i+1}$. A chamber complex is called \textit{thick} if each facet of codimension one is the face of at least three chambers and is called \textit{thin} if each facet of codimension one is the face of exactly two chambers.

\begin{Def}[see {\cite[Definition~1.5.5]{KP-BTbook}}]
	A \textit{building} is a chamber complex $\sB$ equipped with a collection of subcomplexes, called \textit{apartments}, satisfying the following axioms.\index{definition}{building}
	\begin{enumerate}[(i)]
		\item $\sB$ is a thick chamber complex.
		\item Each apartment is a thin chamber complex.
		\item Any two chambers belong to an apartment.
		\item Given two apartments $\sA_1$ and $\sA_2$, and two facets $\sF_1, \sF_2 \in \sA_1 \cap \sA_2$, there exits an isomorphism $\sA_1 \ra \sA_2$ of chamber complexes that leaves invariant $\sF_1, \sF_2$ and all of their faces.
	\end{enumerate}
\end{Def}

\begin{Fact}
	The Bruhat--Tits building $\sB(G, F)$  with its apartments $\sA(T, F)$ attached to maximal split tori $T$ of $G$ is a geometric realization of a building.
\end{Fact}

\subsection{The non-split (tame) case}
We first assume that $G$ splits over an unramified Galois field extension $E$ over $F$. In that case all the above definitions can be descended to $G$ by taking $\Gal(E/F)$-fixed points of the corresponding objects for $G_E$. 
More precisely, we set  \index{notation}{GFxr@$G(F)_{x,r}$}
$$ G(F)_{x,r}=G(E)_{x,r}^{\Gal(E/F)} , $$
where $G(E)_{x,r}$ is defined using the valuation on $E$ that extends the valuation $\val$ on $F$.
As in the split case, we may abbreviate $G(F)_{x,r}$ by $G_{x,r}$. 

Via the action of $\Gal(E/F)$ on $G(E)$ and hence on its filtration subgroups, we obtain an action of $\Gal(E/F)$ on the Bruhat--Tits building $\sB(G,E)$ and we define $$\sB(G,F)=\sB(G,E)^{\Gal(E/F)} . $$

More generally, if we only assume that $G$ splits over a tamely ramified Galois field extension $E$ over $F$, then we have for $r>0$
$$G_{x,r}= G(F)_{x,r}=G(E)_{x,r}^{\Gal(E/F)}  , $$
where $G(E)_{x,r}$ is defined using the valuation on $E$ that extends the valuation $\val$ on $F$ and $U_{\alpha}(E)_{x,r}= x_\alpha(\varpi_E^{\ceil{e(r-\<\alpha, x_{BT}\>)}}\cO_E)$ with $e$ the ramification index of the field extension $E$ over $F$.
Defining the parahoric subgroup $G(F)_{x,0}$ is slightly more subtle in general. It is a finite index subgroup  of $G(E)_{x,0}^{\Gal(E/F)}$. The parahoric subgroup $G(F)_{x,0}$ being occasionally a slightly smaller group than $G(E)_{x,0}^{\Gal(E/F)}$ will ensure that $G(F)_{x,0}/G(F)_{x,0+}$ are the $\bF_q$-points of a connected reductive group rather than a potentially disconnected group. More precisely, the parahoric subgroup $G(F)_{x,0}$ is defined by
$$ G_{x,0}= G(F)_{x,0}=G(E)_{x,0}^{\Gal(E/F)} \cap G(F)^0 $$
for some explicitly constructed normal subgroup $G(F)^0 \subset G(F)$. We refer the interested reader to the literature, e.g., \cite{KP-BTbook}, for the precise definition of $G(F)^0$ and only note that $G(F)^0=G(F)$ if $G$ is simply connected semi-simple, e.g., for $G=\SL_n$ we have $\SL_n(F)^0=\SL_n(F)$.

As in the unramified setting, using the action of $\Gal(E/F)$ on $G(E)$ and hence on its filtration subgroups, we obtain an action of $\Gal(E/F)$ on the Bruhat--Tits building $\sB(G,E)$ and we define $$\sB(G,F)=\sB(G,E)^{\Gal(E/F)} . $$

Similarly, we have for the Lie algebra \index{notation}{gxr@$\fg_{x,r}$} \index{notation}{gFxr@$\fg(F)_{x,r}$}
$$ \fg_{x,r}=\fg(F)_{x,r}=(\fg(E)_{x,r})^{\Gal(E/F)} .$$

We note that the above definitions rely on the extension $E$ over $F$ being tame, but are independent of the choice of $E$.   

\begin{Aside}
	If $G$ splits only over a wildly ramified extension $E/F$, then the space of fixed vectors of the Galois action on the Bruhat--Tits building over $E$ might be larger than the Bruhat--Tits building defined over $F$. We will not introduce the Bruhat--Tits building in that generality in this survey since we mostly restrict to the tame case, and we instead refer the interested reader to the literature, e.g., the original articles by Bruhat and Tits (\cite{BT1, BT2}) and the recent book on this topic by Kaletha and Prasad \cite{KP-BTbook}. 
\end{Aside}

The Moy--Prasad filtration still satisfies the nice properties as in Fact \ref{Fact-BT-triple-properties}, i.e., more precisely
\begin{Fact} \label{Fact-MP-properties} Let $x \in \sB(G,F)$. 
	\begin{enumerate}[(a)] 
		\item $G_{x,r}$ is a normal subgroup of $G_{x,0}$ for all $r \in \bR_{\geq 0}$.
		\item The quotient $G_{x,0}/G_{x,0+}$ is the group of the $\bF_q$-points of a reductive group $\bG_x$ defined over the residue field $\bF_q$ of $F$.
		\item For $r \in \bR_{> 0}$, the quotient $G_{x,r}/G_{x,r+}$ is abelian and can be identified with an $\bF_q$-vector space.
		\item \label{Facts-filtrations-reps} Let $r>0$. Since $G_{x,r}$ is a normal subgroup of $G_{x,0}$, the group $G_{x,0}$ acts on $G_{x,r}$ via conjugation. This action descends to an action of the quotient $G_{x,0}/G_{x,0+}$ on the vector space $G_{x,r}/G_{x,r+}$ and the resulting action is (the $\bF_q$-points of) a linear algebraic action, i.e., corresponds to a morphism from $\bG_x$ to $\GL_{\dim_{\bF_q}(G_{x,r}/G_{x,r+})}$ over $\bF_q$.
		\item \label{MP-isom} Under the assumption that $G$ splits over a tamely ramified field extension of $F$, we have the following ``Moy--Prasad isomorphism'': $G_{x,r}/G_{x,r+} \simeq \fg_{x,r}/\fg_{x,r+}$ for $r \in \bR_{> 0}$ and more general $G_{x,r}/G_{x,s} \simeq \fg_{x,r}/\fg_{x,s}$ for $r, s \in \bR_{> 0}$ with $2r \geq s>r$.
	\end{enumerate}
\end{Fact}

\begin{Def}
	For a maximal split torus $S$ of $G$, we choose a maximal torus $T\subset G$ containing $S$ and call the subset $\sA(S,F):=\sA(T,E)^{\Gal(E/F)}$ of $\sB(G,F)$, sometimes also denoted by $\sA(T,F)$, the \textit{apartment} of $S$ (or of $T$).  \index{definition}{apartment}
\end{Def}
Note that $\sA(S,F)=\sA(T,E)^{\Gal(E/F)}$ is independent on the choice of maximal torus $T$ containing $S$, i.e., the apartment of $S$ is well defined. The apartment $\sA(S,F)$ is an affine space over the real vector space $X_*(S) \otimes \bR/X_*(Z(G))\otimes \bR$. We will equip the apartment with a polysimplicial structure analogous to the split case.

To do this, let $S$ be a maximal split torus of $G$ and $T$ a maximal torus of $G$ containing $S$ that splits over a tame extension $E$ of $F$. We denote by $\Phi(G,S)$ the relative root system, i.e., the restrictions of the (absolute) roots $\Phi(G,T) \subset X^*(T_{E})$ to $S$ that are non-trivial, or in other words the non-trivial weights of the action of $S$ on the Lie algebra of $G$. For $a \in \Phi(G,S)$, we have the root group $U_a$, which is the unique smooth closed subgroup of $G$ that is normalized by $S$ and on whose Lie algebra $S$ acts by positive integer multiples of $a$. We have $U_a(F)=(\prod_\alpha U_\alpha(E))^{\Gal(E/F)}$, where $\alpha$ runs over the roots of $\Phi(G,T)$ that restrict to a positive integer multiple of $a$. We set 
$$U_{a,x,r}:= \left(\prod_\alpha U_\alpha(E)_{x,r}\right)^{\Gal(E/F)}$$
for $r \in \bR_{\geq 0}$ using the same normalization as for the definition of $G_{x,r}$ above.

Now we can define a set of hyperplanes of the apartment $\sA(S, F)$ for $a \in \Phi(G,S)$:
$$ \Psi_a:=\left\{\text{ hyperplanes } H \subset \sA(S,F) \text{ satisfying } \begin{array}{ll}
U_{a,x,0}=U_{a,y,0} & \forall x, y \in H \\
U_{a,x,0} \neq U_{a,x,0+} & \forall x \in H 
\end{array} \right\}  . $$
We set $$ \Psi := \bigcup_{a \in \Phi(G,S)} \Psi_a $$ and use these hyperplanes to turn the apartment $\sA(S,F)$ into the geometric realization of a polysimplicial complex. This polysimplicial structure on the apartments yields a polysimplicial structure on the Bruhat--Tits building $\sB(G,F)$ and as in the split case, we have the following result.
\begin{Fact}
	The Bruhat--Tits building $\sB(G, F)$  with its apartments $\sA(S, F)$ attached to maximal split tori $S$ of $G$ is a geometric realization of a building.
\end{Fact}

We record the following fact that will become useful later when constructing supercuspidal representations.
\begin{Fact}  \label{Fact-x} Let $x, y \in \sA(S,F) \subset \sB(G,F)$.
 Then the image of $G_{x,0} \cap G_{y,0}$ in $G_{y,0}/G_{y,0+}$ is a parabolic subgroup $P_{x,y}$ and the image of $G_{x,0+} \cap G_{y,0}$ in $G_{y,0}/G_{y,0+}$ is the unipotent radical of $P_{x,y}$. If $x \neq y$ and $y$ is a vertex, i.e., a polysimplex of minimal dimension, then $P_{x,y}$ is a proper parabolic subgroup.
\end{Fact}

\subsection{The enlarged Bruhat--Tits building}
In some circumstances it is more convenient to work with the enlarged Bruhat--Tits building. The \textit{enlarged Bruhat--Tits building} $\wt\sB(G,F)$ \index{notation}{BGFtilde@$\wt \sB(G,F)$} \index{definition}{enlarged Bruhat--Tits building} is defined as the product of the reduced Bruhat--Tits building $\sB(G,F)$ with $X_*(Z(G)) \otimes_{\bZ} \bR$, i.e., 
$$ \wt\sB(G,F) = \sB(G,F) \times  X_*(Z(G)) \otimes_{\bZ} \bR .$$
This means that if the center of $G$ is finite, then the reduced and the non-reduced Bruhat--Tits buildings are the same. In general, an important difference is that stabilizers in $G(F)$ of points in the enlarged Bruhat--Tits building are compact while stabilizers of points in the reduced Bruhat--Tits building contain the center of $G(F)$ and are compact-mod-center. For the enlarged building, the apartments $\wt \sA(S,F)$ correspond to maximal split tori $S$ and are affine spaces under the action of $X_*(S)\otimes_\bZ \bR$.
For a point $x \in \wt\sB(G,F)$ we denote by $[x]$ the image of $x$ in $\sB(G,F)$ (by projection to the first factor) and we define $G_{x,r}:=G_{[x],r}$ for $r \in \bR_{\geq 0}$ and $\fg_{x,r}:=\fg_{[x],r}$ and $\fg^*_{x,r}:=\fg^*_{[x],r}$ for $r \in \bR$. \index{notation}{x@$[x]$}

\subsection{The depth of a representation} \label{Section-depth}
The Moy--Prasad filtration allows us to introduce the notion of the depth of an irreducible smooth representation, initially defined by Moy and Prasad in \cites{MP1, MP2}. Our definition is slightly different but equivalent to theirs.
\begin{Def}\label{Def-depth}
	Let $(\pi, V)$ be an irreducible smooth representation of $G$. The \textit{depth} of $(\pi, V)$ is the smallest non-negative real number $r$ such that $V^{G_{x,r+}} \neq \{0\}$ for some $x \in \sB(G,F)$. \index{definition}{depth}
\end{Def}


\section{Construction of supercuspidal representations} \label{Section-construction}
\subsection{A non-exhaustive overview of some historic developments} \label{Section-history-construction}
In 1977, a Symposium in Pure Mathematics was held in Corvallis that led to famous Proceedings. One of the articles in the Proceedings was entitled ``Representations of $p$-adic groups: A survey'', written by Cartier (\cite{Cartier-Corvallis}). We quote from the introduction of this article:  
\begin{quote}
	``The main goal of this article will be the description and study of the principal series and the spherical functions. There shall be almost no mention of two important lines of research which are still actively pursued today: \\
	(a) [...] \\	
	(b) \textit{Explicit construction of absolutely cuspidal representations [nowadays usually called ``supercuspidal representations'']}. Here important progress has been made by Shintani \cite{Shintani68}, Gérardin \cite{Gerardin75} and Howe (forthcoming papers in the Pacific J. Math.). One can expect to meet here difficult and deep arithmetical questions which are barely uncovered.''
\end{quote}	
Since then, mathematicians have tried to construct the mysterious supercuspidal representations.  To mention a few, in 1979, Carayol (\cite{Carayol}) gave a construction of all supercuspidal representations of the general linear group $\GL_n(F)$ for $n$ a prime number and $F$ of characteristic zero. In 1986, Moy (\cite{Moy-exhaustion}) proved that Howe's construction (\cite{Howe}) from the 1970s exhausts all supercuspidal representations of $\GL_n(F)$ if $n$ is a positive integer coprime to $p$. In the early 1990s, Bushnell and Kutzko extended these constructions to obtain all supercuspidal representations of $\GL_n(F)$ and $\SL_n(F)$ for arbitrary $n$ (\cite{BK, BK-SLn}).  
Similar methods have been exploited by Stevens (\cite{Stevens}) around 15 years ago to construct all supercuspidal representations of classical groups for $p \neq 2$, i.e., orthogonal, symplectic and unitary groups. His work was preceded by a series of partial results by Moy (\cite{MoyU21} for $\U(2,1)$, \cite{MoyGSp4} for $\GSp_4$), Morris (\cite{Morris1} and \cite{Morris2}) and Kim (\cite{Kimreps}).
Moreover, Zink (\cite{Zink}) treated division algebras over non-archimedean local fields of characteristic zero,  Broussous (\cite{Broussous}) treated division algebras without restriction on the characteristic, and S\'echerre and Stevens (\cite{Secherre-Stevens}) completed the case of all inner forms of $\GL_n(F)$ about 15 years ago.

In order to achieve progress for arbitrary reductive groups, the work of Moy and Prasad based on the work of Bruhat and Tits, introduced in Section \ref{Section-MP-BT}, was pivotal. 
The Moy--Prasad filtration allowed Moy and Prasad introduced in \cite{MP1, MP2} to introduce the notion of \textit{depth} of a representation, see Definition \ref{Def-depth}, and gave a classification of depth-zero representations. Moy and Prasad showed, roughly speaking, that depth-zero representations  correspond to representations of finite groups of Lie type. A similar result was obtained shortly afterwards using different techniques by Morris (\cite{Morrisdepthzero}). We will discuss depth-zero representations in more detail in Section \ref{Section-depth-zero}.

In 1998, Adler (\cite{Adler})  used the Moy--Prasad filtration to suggest a construction of positive-depth supercuspidal representations for general $p$-adic groups that split over a tamely ramified extension of $F$, which was generalized by Yu (\cite{Yu}) in 2001. Since then, Yu's construction has been widely used in the representation theory of $p$-adic groups as well as for applications thereof. 
We will sketch Yu's construction in Section \ref{Section-Yus-construction}.

Kim (\cite{Kim}) achieved the subsequent breakthrough in 2007 by proving that if $F$ has characteristic zero and the prime number $p$ is ``very large'', then all supercuspidal representations arise from Yu's construction. Recently, in 2021, Fintzen (\cite{Fi-exhaustion})  has shown via very different techniques that Yu's construction provides us with all supercuspidal representations only under the minor assumption that $p$ does not divide the order of the absolute Weyl group of the (tame) $p$-adic group. In particular, the result also holds for fields $F$ of positive characteristic. 
\begin{table}[h]
	\begin{tabular}{|c|c|c|c|c|c|}
		\hline type  & 			$A_n \, (n \geq 1)$     & $B_n, C_n \, (n \geq 2)$  &  $D_n \, (n \geq 3)$              & $E_6$                 &  $E_7$                   \\ 
		\hline order &  $(n+1)!$  & $2^n \cdot n!$   &  $2^{n-1} \cdot n!$ & $2^7\cdot3^4\cdot5$  &  $2^{10}\cdot3^4\cdot5\cdot7$\\ 
		\hline 
	\end{tabular} \\
	
	\begin{tabular}{|c|c|c|c|}
		\hline type  & 		  $E_8$ & $F_4 $ & $G_2$  \\ 
		\hline order &   $2^{14}\cdot3^5\cdot5^2 \cdot 7$ &  $2^7 \cdot 3^2$ & $2^2 \cdot 3$ \\ 
		\hline 
	\end{tabular}
	
	\caption{Order of irreducible Weyl groups (\cite[VI.4.5-VI.4.13]{Bourbaki-4-6})}
	\label{table-weyl-group}
\end{table}
Based on \cite{Fi-tame-tori}, we expect this result to be essentially optimal (when considering also types for non-supercuspidal Bernstein blocks and treating all inner forms together), and it is exciting research in progress to construct the remaining supercuspidal representations for small primes. For this survey, we  will focus on the known construction of supercuspidal representations under the assumption that $p$ does not divide the order of the absolute Weyl group.

While Yu's construction attaches to a given input (spelled out in Section \ref{Section-input}) a supercuspidal representation, Hakim and Murnaghan (\cite{HakimMurnaghan}), together with a result of Kaletha (\cite{Kaletha-regular}) that removes some assumptions in the former work, answered the questions of which inputs yield the same supercuspidal representations. We will explain this result in Section \ref{Section-HM}, which thus leads to a parametrization of supercuspidal representations. However, it was recently suggested by Fintzen, Kaletha and Spice (\cite{FKS}) to twist Yu's construction by a quadratic character, i.e., a character of an appropriate compact open subgroup appearing in the construction of supercuspidal representations that takes values in $\{\pm 1\}$. While on first glance  this just looks like changing the parametrization of supercuspidal representations, the existence of the quadratic character has far-reaching consequences. For example, it allowed to calculate formulas for the Harish-Chandra character of these supercuspidal representations (\cite{FKS, Spice21v2}), to specify a candidate for the local Langlands correspondence for non-singular supercuspidal representations (\cite{Kaletha-non-singular}) and to prove that the local Langlands correspondence for regular supercuspidal representations introduced by Kaletha (\cite{Kaletha-regular}) satisfies the desired character identities (\cite{FKS}). It is also crucial for obtaining isomorphisms between Hecke algebras attached to Bernstein blocks of arbitrary depth and those of depth-zero (\cite{FOAMv1}), a topic that we will not discuss in this survey. We will introduce the quadratic character in Section \ref{Section-epsilon} and provide more details on Harish-Chandra characters in Section \ref{Section-characters}.

In 2019, Kaletha (\cite{Kaletha-regular}) observed that a large subclass of the representations constructed by Yu, which Kaletha called \textit{regular supercuspidal representations} arise from much simpler data, consisting of only an elliptic maximal torus and a character thereof satisfying certain properties and that these representation show a surprising parallel to a large class of (essentially) discrete series representations of real reductive Lie groups, as we hinted to in Section \ref{Sect-motivation-and-real} and will explain in more details in Section \ref{Section-regular-supercuspidal}.

\subsection{Generalities about the construction of supercuspidal representations} \label{Section-depth-zero}
We refer the reader to the article by Taïbi (\cite{Taibi-IHES}) that appears in the same proceedings as this article, in particular Section~3.1 and 3.2, and to \cite[Section~2.4]{Fi-CDM} for the basic notions around the smooth representations of our $p$-adic group $G(F)$ including parabolic induction and the notion of \text{supercuspidal}. (Note that we use the expression ``$p$-adic group'' to refer to $G(F)$ independent of the characteristic of the non-archimedean local field $F$.) All representations are always taken to be smooth and to have complex coefficients unless stated otherwise.  Possible references for the facts discussed below include \cite{BH-GL2, DeBacker-notes, Renard-book, Vigneras-book}.

It is a folklore conjecture that all supercuspidal irreducible representations arise via compact induction from a representation of a compact-mod-center open subgroup of $G(F)$, and all constructions mentioned above proceed in this way. In \cite[\S3.2]{Taibi-IHES} we saw different equivalent definitions of when a smooth irreducible representation is called \textit{supercuspidal}. One characterization is to ask that all the matrix coefficients are compactly supported modulo center. Using that viewpoint it is a nice exercise to deduce the following lemma.
\begin{Lemma} \label{Lemma-sc}
	Let $K$ be a compact-mod-center open subgroup of $G(F)$ and let $\rho$ be an irreducible representation of $K$. 
	If the compact induction $\cind_{K}^{G(F)} \rho$ of $\rho$ from $K$ to $G(F)$ is irreducible, then $\cind_{K}^{G(F)} \rho$ is a supercuspidal representation of $G(F)$.
\end{Lemma}

Thus in order to construct supercuspidal representations, it suffices to construct pairs $(K, \rho)$ of compact-mod-center open subgroups and irreducible representations thereof such that $\cind_{K}^{G(F)} \rho$ is irreducible. The standard approach to show the latter is via Lemma \ref{Lemma-intertwiner} below, which we will demonstrate in examples below. In order to state the fact, we need to introduce some notation. 

Let $K$ be a compact-mod-center open subgroup of $G(F)$ that contains the center $Z(G(F))$ of $G(F)$ and let $(\rho, W)$ be a smooth representation of $K$.
\begin{Not}
	For $g \in G(F)$, we write ${^g}\rho$ for the representation of ${^g}K:=gKg^{-1}$ satisfying ${^g}\rho(h)=\rho(g^{-1}hg)$ for $h \in {^g}K$. 
	
	We say that $g$ \textit{intertwines} $(\rho, W)$ if $\Hom_{{^g}K \cap K}({^g}\rho|_{{{^g}K \cap K}}, \rho|_{{{^g}K \cap K}}) \neq \{ 0\}$. \index{definition}{intertwine}
\end{Not}

\begin{Lemma} \label{Lemma-intertwiner}
	Let $K$ be an open subgroup of $G(F)$ that contains and is compact modulo the center $Z(G(F))$ of $G(F)$. Let $(\rho, W)$ be an irreducible representation of $K$.
	Suppose $g \in G(F)$ intertwines $(\rho, W)$ if and only if $g \in K$. Then $\cind_{K}^{G(F)} \rho$ is irreducible.
\end{Lemma}

In order to prove this lemma, let us note a helpful result, the Mackey decomposition, whose proof is a nice exercise using the definition of the compact induction.

\begin{Lemma}[Mackey decomposition] \label{Lemma-Mackey}
	If $K'$ is a compact-mod-center open subgroup of $G(F)$, then the restriction of $\cind_{K}^{G(F)} \rho$ to $K'$ decomposes as a representation of $K'$ as follows
	$$(\cind_{K}^{G(F)} \rho) |_{K'}= \bigoplus_{g \in  K'\backslash G(F) /K} \cind_{{^g}K \cap K'}^{K'}{^g}\rho|_{{^g}K \cap K'} \, \, . $$	
\end{Lemma} \index{definition}{Mackey decomposition}
\begin{proof}
	Left to the reader.
\end{proof}

\begin{proof}[Proof of Lemma \ref{Lemma-intertwiner}]
First note that $(\rho, W)$ is a $K$-subrepresentation of $\left(\left(\cind_K^{G(F)}\rho\right)|_{K},\right.$ $\left.\cind_K^{G(F)} W\right)$ via the embedding
\[ 
w \ \ \mapsto \ \  f_w: g \mapsto \begin{cases}
	\rho(g)w & g \in K\\
	0 & g \notin K
\end{cases},
\]
and the image of $W$ in $\cind_K^{G(F)} W$ under this embedding generates the latter as a $G(F)$-representation. We claim that this is up to scalar the only embedding of $\rho$ into $\left(\cind_K^{G(F)}\rho\right)|_{K}$. This follows from:
\begin{eqnarray*}
\Hom_K\left(\rho, (\cind_{K}^{G(F)} \rho) |_{K}\right) 
&\simeq &  \bigoplus_{g \in K\backslash G(F) /K} \Hom_K\left(\rho, \cind_{{^g}K \cap K}^{K}{^g\rho}|_{{^g}K \cap K}\right) \\
&\simeq &  \bigoplus_{g \in K\backslash G(F) /K} \Hom_{{^g}K \cap K}\left(\rho|_{{^g}K \cap K}, {^g\rho}|_{{^g}K \cap K}\right) \\
& = &  \Hom_{K}\left(\rho|_{K \cap K}, {\rho}|_{K \cap K}\right) \simeq \bC,
\end{eqnarray*}
where the first isomorphism results from the Mackey decomposition (Lemma \ref{Lemma-Mackey}) and the second from Frobenius reciprocity as the involved compact induction agrees with the smooth induction.

Suppose now that $V \subset \cind_K^{G(F)}W$ is a non-trivial $G(F)$-subrepresentation. This implies
\begin{equation*}
	\{ 0 \}  \not\simeq   \Hom_G\left(V, \cind_{K}^{G(F)} W\right) \subset \Hom_G\left(V, \Ind_{K}^{G(F)} W\right) 
	\simeq  \Hom_K(V, W) ,
\end{equation*}
where $\Ind$ denotes the smooth induction.
Note that $Z(G(F))$ acts via the central character of $\rho$ on $\cind_K^{G(F)}W$ and hence on $V$. Thus, as a $K$-representation, $V$ is a direct sum of irreducible $K$-representations. Therefore the above observation $\Hom_K(V, W) \not \simeq \{ 0 \}$ implies that $W$ is isomorphic to a subrepresentation of $V$. By the uniqueness up to scalar of the embedding of $W$ into $\cind_K^{G(F)}W$ as $K$-representations, we deduce that $V$ contains the above image of $W$, which generates $\cind_K^{G(F)}W$ as a $G(F)$-representation. Since $V$ is a $G(F)$-representation, we obtain $V=\cind_K^{G(F)}W$.
\end{proof}

\subsection{Depth-zero supercuspidal representations} \label{Section-depth-zero}
In this section, we consider the special case of depth-zero supercuspidal representations. The following theorem is due to Moy and Prasad (\cite{MP1, MP2}) and a different proof was later given by Morris (\cite{Morrisdepthzero}).

\begin{Thm}[\cite{MP1, MP2, Morrisdepthzero}] \label{Thm-depth-zero}
	Let $x \in \sB(G,F)$ be a vertex. Let $(\rho, V_\rho)$ be an irreducible smooth representation of the stabilizer $G_x$ of $x$ that is trivial on $G_{x,0+}$ and such that $\rho|_{G_{x,0}}$ is a cuspidal representation of the reductive group $G_{x,0}/G_{x,0+}$. Then  $\cind_{G_x}^{G(F)} \rho$ is a supercuspidal irreducible representation of $G(F)$.
\end{Thm}

The above authors also showed that all depth-zero supercuspidal (irreducible smooth) representations are of the form as in Theorem \ref{Thm-depth-zero}.

\textbf{Proof of Theorem \ref{Thm-depth-zero}.}\\
By Lemmata \ref{Lemma-sc} and \ref{Lemma-intertwiner} it suffices to show that an element $g \in G(F)$ intertwines $(\rho, V_\rho)$ if and only if $g \in G_x$.
Since all $g \in G_x$ intertwine $(\rho, V_\rho)$, it remains to show the other direction of the implication. Hence we assume $g \in G(F)$ intertwines $(\rho, V_\rho)$,  i.e., we can choose a nontrivial element
$$ f \in \Hom_{G_x \cap gG_xg^{-1}}({^g\sigma},\sigma) \not\simeq \{0\} .$$
Since $\sigma$ is trivial when restricted to $G_{x,0+}$, the representation ${^g\sigma}$ is trivial when restricted to $gG_{x,0+}g^{-1}=G_{g.x,0+}$. Hence $G_{g.x,0+} \cap G_{x,0}$ acts trivially on the image $\Image(f)$ of $f$. If $g \notin G_x$, then $g.x \neq x$ and hence by Fact \ref{Fact-x}, the image of $G_{g.x,0+} \cap G_{x,0}$ in $G_{x,0}/G_{x,0+}$ is the unipotent radical $N$ of a proper parabolic subgroup of $G_{x,0}/G_{x,0+}$. Thus  
$$ \{0 \} \not\simeq \Image(f) \subset V_\rho^N, $$
which contradicts that $(\rho, V_\rho)$ is cuspidal. \qed

\subsection{An example of a positive-depth supercuspidal representations} \label{Section-simple-sc-example} 
From now on we fix an additive character $\varphi: F \rightarrow \bC^\times $ \index{notation}{phi@$\varphi$} (i.e., a group homomorphism from the group $F$ (equipped with addition) to the group $\bC^\times$ (equipped with multiplication)) that is nontrivial on $\cO$ and trivial on $\varpi\cO$.

We start with an example of a positive-depth supercuspidal representation that we will step-wise generalize.
 Let $G=\SL_2$. Consider the point $x_2 \in \sB(\SL_2,F)$ introduced in Example 2 on page \pageref{Explx2}, which is the unique point $x_2$ for which  
\[G_{x_2,r}=
\begin{pmatrix}
	1+\varpi^{\ceil{r}}\cO & \varpi^{\ceil{r-\frac{1}{2}}}\cO \\  \varpi^{\ceil{r+\frac{1}{2}}}\cO & 1+\varpi^{\ceil{r}}\cO
\end{pmatrix}_{\det=1} 
\quad
\text{for} 
\, \, r \in \bR_{>0} .
\]
Let $K=\{\pm \Id \}G_{x_2, \frac{1}{2}}$. We define the representation $(\rho, \bC)$, i.e., the morphism $\rho: K \ra \bC^\times$ by requiring
\[\rho (\pm \Id)=1 
\quad
\text{ and }
\quad
\rho \left(
\begin{pmatrix}
	1+\varpi a & b \\  \varpi c & 1+\varpi d
\end{pmatrix}
\right)
=
\varphi(b+c)
\]
for all $a, b, c, d \in \cO$ with $(1+\varpi a)(1 + \varpi d)-\varpi bc=1$. Note that $\rho$ is trivial on $G_{x_2, \frac{1}{2}+}$, i.e., factors through $K/G_{x_2, \frac{1}{2}+}$.
\begin{Fact}
The representation $\cind_K^{\SL_2(F)} \rho$ is a supercuspidal irreducible representation of depth $\frac{1}{2}$.
\end{Fact}
If $p\neq 2$, this is a very special case of Yu's construction as we will see below. This construction also works for $p=2$ and is an example of a \textit{simple supercuspidal representation} introduced by Gross and Reeder (\cite[\S9.3]{GrossReeder}), which in turn are special cases of \textit{epipelagic representations} as introduced by Reeder and Yu (\cite{ReederYu}), which are representations of smallest positive depth. More precisely, for $x \in \sB(G, F)$, let $r(x)$ be the smallest positive real number for which $G_{x,r(x)}\neq G_{x,r(x)+}$. Then an irreducible representation $(\pi, V)$ is called \textit{epipelagic} if there exists $x \in \sB(G,F)$ such that $V^{G_{x,r(x)+}}$ is non-trivial and $(\pi, V)$ has depth $r(x)$ (\cite[\S2.5]{ReederYu}).

\subsection{Generic characters} \label{Section-generic-characters}
In order to generalize the example of the previous subsection and to eventually present Yu's general construction of supercuspidal representations, we need the notion of twisted Levi subgroups and generic characters.

\begin{Def}
	A subgroup $G'$ of $G$ is a \textit{twisted Levi subgroup} if $G'_E$ is a Levi subgroup of $G_E$ for some finite field extension $E$ over $F$. \index{definition}{twisted Levi subgroup}
\end{Def}
If $G'$ is a twisted Levi subgroup of $G$, and we assume that $G'$ splits over a tamely ramified field extension of $F$, then we have an embedding of the enlarged Bruhat--Tits building $\wt\sB(G',F)$ of $G'$ into the enlarged Bruhat--Tits building $\wt\sB(G,F)$ of $G$. This embedding is unique up to translation by $X_*(Z(G')) \otimes_{\bZ} \bR$, and its image is independent of the embedding. We will fix such embeddings when working with twisted Levi subgroups to view $\wt\sB(G',F)$ as a subset of $\wt\sB(G,F)$.

In order to define generic characters (following \cite[\S2.1]{Fi-mod-ell}, which is based on \cite[\S9]{Yu}, but is slightly more general for small primes, see \cite[Remark~2.2]{Fi-mod-ell} for details), we first define the notion of generic elements in the dual of the Lie algebra and then use the Moy--Prasad isomorphism to obtain the notion of generic characters.

We denote by $\Phi(G, T)$\index{notation}{Phi(G,T)@$\Phi(G,T)$} the absolute root system of $G$ with respect to $T$, i.e., the roots of $G_{\bar F}$ with respect to $T_{\bar F}$, where $\bar F$ denotes a separable closure of $F$. We also extend the valuation $\val$ on $F$ to a valuation $\val: \bar F \rightarrow \bQ \cup \{ \infty \}$\index{notation}{val} on $\bar F$ and denote by $\cO_{\bar F}$ all the elements of $\bar F$ with non-negative or infinite valuation.

Let $G' \subseteq G$ be a twisted Levi subgroup that splits over a tamely ramified field extension of $F$, and denote by $(\Lie^*(G'))^{G'}$ the subscheme of the linear dual of the Lie algebra $\Lie(G')$ of $G'$ fixed by (the dual of) the adjoint action of $G'$. 
\begin{Def}  Let $x \in \wt \sB(G',F)$ and $r \in \bR_{> 0}$.
	\begin{enumerate}[(a)]
		\item 	An element $X$ of $(\Lie^*(G'))^{G'}(F) \subset \Lie^*(G')(F)$ is called \textit{$G$-generic of depth $r$} (or \textit{$(G, G')$-generic of depth $r$}) if the following three conditions hold. 
		\begin{itemize}
			\item[\textbf{(GE0)}] For some (equivalently, every) point $x \in \wt\sB(G', F)$, we have $X \in \Lie^*(G')_{x,r} \smallsetminus \Lie^*(G')_{x,r+}$.
			\item[\textbf{(GE1)}] $\val(X(H_\alpha))=r $ for all $\alpha \in \Phi(G, T) \smallsetminus \Phi(G',T)$ for some (equivalently, every) maximal torus $T$ of $G'$, where $H_\alpha:=d\check\alpha(1) \in \fg(\bar F)$ with $d\check\alpha$ the derivative of the coroot $\check \alpha \in X_*(T_{\bar F})$ of $\alpha$. 
			\item[\textbf{(GE2)}] \textbf{GE2} of \cite[\S8]{Yu} holds, which we recall below and which is implied by (GE1) if $p$ does not divide the order of the absolute Weyl group of $G$. 
		\end{itemize}
		\item 	A character $\phi$ of $G'(F)$ is called \textit{$G$-generic (or $(G, G')$-generic) relative to $x$ of depth $r$} if $\phi$ is trivial on $G'_{x,r+}$ and the restriction of $\phi$ to $G'_{x,r}/G'_{x,r+}\simeq \fg'_{x,r}/\fg'_{x,r+}$ is given by $\varphi \circ X$ for some element $X \in (\Lie^*(G'))^{G'}(F)$ that is $(G, G')$-generic of depth $-r$. \index{definition}{generic character} \index{definition}{genericG@$G$-generic} \index{definition}{genericGG'@$(G, G')$-generic}
	\end{enumerate}
\end{Def}
The equivalence in (GE0) is proven in \cite[Lemma~2.3]{Fi-mod-ell}.

In order to explain Condition (GE2), let $X \in (\Lie^*(G'))^{G'}(F)\subset \Lie^*(G')(F)$ satisfy (GE0) and (GE1) for a maximal torus $T$ of $G'$. We denote by $X_\ft$ the restriction of $X$ to $\ft(\bar F)$ and choose an element $\varpi_r$ of valuation $r$ in $\bar F$. Then, under the identification of $\ft^*(\bar F)$ with $X^*(T_{\bar F}) \otimes_\bZ \bar F$, the element $\frac{1}{\varpi_r} X_\ft$ is contained in  $X^*(T_{\bar F}) \otimes_\bZ \cO_{\bar F}$, and we denote its image under the surjection $X^*(T_{\bar F}) \otimes_\bZ \cO_{\bar F} \twoheadrightarrow X^*(T_{\bar F}) \otimes_\bZ \bar \bF_q$ by $\bar X_\ft$. Now we can state Condition (GE2):
\begin{itemize}
	\item[\textbf{(GE2)}] The subgroup of the absolute Weyl group of $G$ that fixes $\bar X_\ft$ is the absolute Weyl group of $G'$.
\end{itemize}

\begin{Rem}
	By \cite[Lemma~8.1]{Yu}, Condition \textbf{(GE1)} implies \textbf{(GE2)} if $p$ is not a torsion prime for the dual root datum of $G$, i.e., in particular, if $p$ does not divide the order of the absolute Weyl group of $G$. 
\end{Rem}

\begin{Rem}
	It is work in progress 
	to construct supercuspidal representations for a more general notion of ``generic'' that does not require (GE2) to be satisfied (and only requires a weaker version of (GE1)).	
\end{Rem}

\begin{Rem}
	If a character $\phi$ of $G'(F)$ is $(G, G')$-generic relative to $x$ of depth $r$, then it is also $(G, G')$-generic relative to $x'$ of depth $r$ for every $x' \in \wt \sB(G',F)$ , i.e., the notion of genericity does not depend on the choice of point $x$ (\cite[Lemma~3.3.1]{FOAMv1}). 
\end{Rem}

\begin{Rem}
	We caution the reader that an element in $\Lie^*(G')(F)$ that is $G$-generic of depth $r$ is sometimes called ``$G$-generic of depth $-r$'' in the literature (e.g., in \cite[\S8]{Yu} and \cite[\S2.1]{Fi-mod-ell}). However, such an element has depth $r$, in the sense of it being contained in $\Lie^*(G')_{x,r} \smallsetminus \Lie^*(G')_{x,r+}$, and therefore the latter convention has led to some confusion in the literature in the past. 
\end{Rem}

\begin{Rem}
	Usually the notion of ``$(G, G')$-generic'' is only defined for $G' \subsetneq G$. However, sometimes it is convenient to also consider the case $G'=G$, see, e.g., \cite{FOAMv1}, and in this case our definition implies that a $(G, G)$-generic character of depth $r$ has indeed depth $r$. In particular, we do not consider the trivial character a $(G, G)$-generic character of depth $r$. This differs from Yu's convention in \cite[\S~15, p.\ 616]{Yu} where he considers trivial characters as $G$-generic of depth $r$. We have chosen the above more restrictive definition of $(G, G)$-generic characters of depth $r$ as it allows to construct more uniformly representations of depth $r$ from a $(G, G')$-generic character of depth $r$ without having to treat the case $G=G'$ separately.
\end{Rem}

To provide some examples of generic characters, we consider the case that $F=\bQ_{7}$, $G=\GL_2$ and $G'$ is the diagonal torus $T \subset \GL_2$. We let $\psi: \bQ_{7}^\times \ra \bC^\times$ be a character of depth 1. 
 Then the following three characters
\[\begin{pmatrix} t_1  & 0 \\ 0 & t_2  \end{pmatrix} \mapsto \psi(t_1) \quad 
\text{and } \quad
\begin{pmatrix} t_1  & 0 \\ 0 & t_2  \end{pmatrix} \mapsto \psi(t_2) \quad
\text{ and } \quad
 \begin{pmatrix} t_1  & 0 \\ 0 & t_2  \end{pmatrix} \mapsto \psi(t_1t_2^{-1})\]
are $(G, T)$-generic of depth 1 relative to any point $x\in \wt \sB(T,\bQ_7) \subset \wt \sB(G,\bQ_7)$.
The two characters 
\[\begin{pmatrix} t_1  & 0 \\ 0 & t_2  \end{pmatrix} \mapsto \psi(t_1t_2) \quad 
\text{and } \quad
\begin{pmatrix} t_1  & 0 \\ 0 & t_2  \end{pmatrix} \mapsto \psi(t_1t_2^{-6})\]
are also of depth 1 relative to any point $x\in \wt \sB(T,\bQ_7) \subset \wt \sB(G,\bQ_7)$, but they are not $(G,T)$-generic of depth 1.

\subsection{More examples of positive-depth supercuspidal representations} \label{Section-more-positive-depth-sc} 
We will now use generic characters to provide a construction of supercuspidal representations of positive depth that generalizes the example provided in Section \ref{Section-simple-sc-example} and have arbitrary large depth. As input for the construction we take the following data
\begin{enumerate}[(a)]
	\item $S \subset G$ an elliptic maximal torus of $G$ that splits over tamely ramified extension $E$ of $F$,
	\item $x \in \wt\sB(S,F)\subset \wt\sB(G,F)$, 
	\item $r \in \bR_{>0}$ such that $G_{x,\frac{r}{2}}=G_{x,\frac{r}{2}+}$,
	\item $\phi: S(F) \ra \bC^\times$ a character that is $(G,S)$-generic relative to $x$ of depth $r$.
\end{enumerate}
The supercuspidal representation that we construct from this input is of the form $\cind_K^{G(F)} \hat \phi$ with 
	$K= S(F)G_{x,\frac{r}{2}}$ and $\hat \phi$ the extension of $\phi$ obtained by ``sending the root groups to 1''. More precisely, $\hat \phi$ is the unique character of $S(F)G_{x,\frac{r}{2}}$ that satisfies \index{notation}{phihat@$\hat \phi$}
	\begin{enumerate}[(i)]
		\item $\hat \phi|_{S(F)}=\phi$, and
		\item $\hat \phi|_{G_{x,\frac{r}{2}}}$ factors through 
		\begin{eqnarray*} 
			G_{x,\frac{r}{2}}= G_{x,\frac{r}{2}+} &\twoheadrightarrow & G_{x,\frac{r}{2}+}/G_{x,r+}  \simeq \fg_{x,\frac{r}{2}+}/\fg_{x,r+} = (\fs(F) \oplus \fr(F))_{x,\frac{r}{2}+}/(\fs(F) \oplus \fr(F))_{x,r+} \\
			& \twoheadrightarrow & \fs_{x,\frac{r}{2}+}/\fs_{x,r+} \simeq 
			S_{x,\frac{r}{2}+}/S_{x,r+},
		\end{eqnarray*}
		on which it is induced by $\phi|_{S_{x,\frac{r}{2}+}}$, where the subspace $\fr(F)$ of root subspaces is defined to be 
		$$\fr(F)= \fg(F) \cap \bigoplus_{\alpha \in \Phi(G,S)} \fg(E)_\alpha,$$ and the surjection $\fs(F) \oplus \fr(F) \twoheadrightarrow \fs(F)$ sends $\fr(F)$ to zero. The isomorphisms used are the Moy--Prasad isomorphisms from Fact \ref{Fact-MP-properties}\eqref{MP-isom}.
	\end{enumerate}
	
\begin{Fact}
	The representation $\cind_{S(F)G_{x,\frac{r}{2}}}^{G(F)} \hat \phi$ is a supercuspidal irreducible representation of depth $r$.
\end{Fact}	
The construction of these representations is a special case of the construction of supercuspidal representations provided by Adler (\cite{Adler}) that was later generalized by Yu (\cite{Yu}). (These references impose a condition on $p$, but this is not necessary for the above special case.)

We recover the representation constructed in Section \ref{Section-simple-sc-example} under the assumption that $p \neq 2$ from the following input 
\begin{enumerate}[(a)]
	\item  $S \subset \SL_2$ is the torus that satisfies for every field extension $F'$ of $F$
	$$ S(F') =\left\{ \begin{pmatrix}
		a & b \\ \varpi b & a 
	\end{pmatrix} \in \SL_2(F') \, | \, a, b \in F' \right\}  .$$
	Then $S$ splits over the quadratic extension $F(\sqrt{\varpi})$ of $F$. 
	\item  The Bruhat--Tits building $\sB(\SL_2, F)=\wt\sB(\SL_2, F)$ is an infinite tree of valency $\abs{\bF_q}+1$ and the Bruhat--Tits building $\sB(S,F)=\wt \sB(S,F)$ of $S$ is a single point that embeds into $\sB(\SL_2, \bQ_p)$ as the barycenter $x$ of an edge, see Figure \ref{Figure-BTtree}. Hence there is a unique choice for $x \in \wt\sB(S,F)\subset \wt\sB(G,F)$.
	\item We let $r= \frac{1}{2}$.
	\item We define $\phi:S(F) \ra \bC^\times$ by 
	\[
	\phi\left( \begin{pmatrix}
		a & b \\ \varpi b & a 
	\end{pmatrix}  \right)
	= \varphi(2ab).
	\] Then $\phi$ is $(\SL_2,S)$-generic of depth $\frac{1}{2}$.
\end{enumerate}
\begin{figure}[h]
	\centering
	\includegraphics[width=0.50\textwidth]{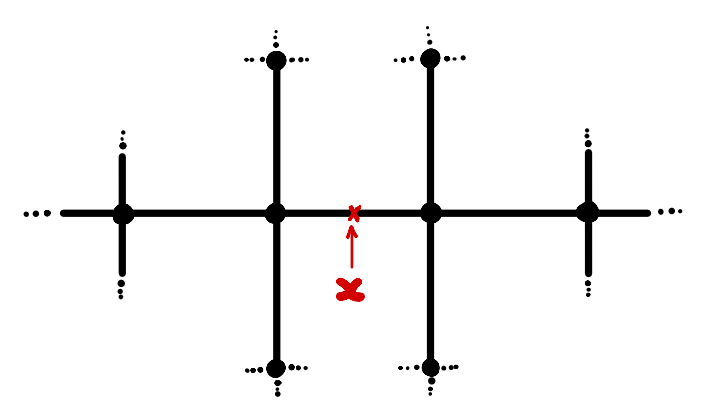}
	\caption{Excerpt of the Bruhat--Tits building $\sB(\SL_2, \bQ_3)$}
	\label{Figure-BTtree}
\end{figure}

\begin{Rem} \label{Rem-input-for-reps}
	Since $S$ is an elliptic maximal torus of $G$, the building $\wt\sB(S,F)$ is equal to $x+X_*(S)\otimes_\bZ \bR=x+X_*(Z(G))\otimes_\bZ \bR$, where we recall that $X_*(?)=\Hom_F(\bG_m, ?)$, and hence the choice of $x \in \wt\sB(S,F)$ has no influence on the construction. Moreover, the real number $r$ is just the depth of $\phi$, i.e., can be read off from $\phi$. Thus, the actual input for the above construction consists only of the pair $(S, \phi)$. 
\end{Rem}

We will now generalize this construction to allow the case $G_{x,\frac{r}{2}} \neq G_{x,\frac{r}{2}+}$, which Yu has dealt with using the theory of Weil--Heisenberg representations and which is why he assumes $p \neq 2$, and to allow a more general sequence of twisted Levi subgroups instead of only $S \subset G$.

\subsection{The input for the construction by Yu} \label{Section-input}
We assume from now on that $p \neq 2$. For a generalization of the below construction of supercuspidal representations that also works if $p=2$ we refer the reader to \cite{Fi-Schwein-Yu2}.

The input for the construction of supercuspidal representations by Yu (following the notation of \cite{Fi-Yu-works})  is a tuple $((G_i)_{1 \leq i \leq n+1}, x, (r_i)_{1 \leq i \leq n}, \rho, (\phi_i)_{1 \leq i \leq n})$ for some non-negative integer $n$ where
\begin{enumerate}[(a)]
	\item $G=G_1 \supseteq G_2 \supsetneq G_3 \supsetneq \hdots \supsetneq G_{n+1}$ are twisted Levi subgroups of $G$ that split over a tamely ramified extension of $F$,
	\item  $x \in \wt\sB(G_{n+1},F)\subset \wt\sB(G,F)$, 
	\item $r_1 > r_2 > \hdots > r_n >0$ are real numbers,
	\item $\phi_i$, for $1 \leq i \leq n$, is a character (i.e., a one-dimensional representation) of $G_{i+1}(F)$ of depth $r_i$, 
	\item $\rho$ is an irreducible representation of $(G_{n+1})_{[x]}$ that is trivial on $(G_{n+1})_{x,0+}$, 
\end{enumerate}
satisfying the following conditions 
\begin{enumerate}[(i)]
	\item  $Z(G_{n+1})/Z(G)$ is anisotropic, i.e., its $F$-points are a compact group,
	\item the image $[x]$ of the point $x$ in $\sB(G_{n+1},F)$ is a vertex, i.e., a polysimplex of minimal dimension,
	\item $\phi_i$ is $(G_i, G_{i+1})$-generic relative to $x$  of depth $r_i$ for all $1 \leq i \leq n$ with $G_i \neq G_{i+1}$, 
	\item $\rho|_{(G_{n+1})_{x,0}}$ is a cuspidal representation of the reductive group $(G_{n+1})_{x,0}/(G_{n+1})_{x,0+}$.
\end{enumerate}
We will call a tuple satisfying the above conditions a \textit{Yu datum}. \index{definition}{Yu datum}

\begin{Aside}
	Our conventions for the notation (following \cite{Fi-Yu-works}) differ slightly from those in \cite{Yu}. In particular, Yu's notation for the twisted Levi sequence is $G^0 \subsetneq G^1 \subsetneq G^2 \subsetneq \hdots \subsetneq G^d$. The reader can find a translation between the two different notations in \cite[Remark~2.4]{Fi-Yu-works}.
\end{Aside}

\textbf{Example of a Yu datum.} We provide an example of a Yu datum for the group $G=\SL_2$  with $p$ an odd prime. We let $n=1$.
\begin{enumerate}[(a)]
	\item  We have $G_1=G$ and let $G_2=S$ be the non-split torus $S \subset \SL_2$ that satisfies 
	$$ S(F') =\left\{ \begin{pmatrix}
	a & b \\  \varpi b  & a 
	\end{pmatrix} \in \SL_2(F') \, | \, a, b \in F' \right\} \, \text{ for all field extensions $F'$ of $F$.}$$
	\item  The point $x$ is the unique point of $\wt\sB(S,F)\subset \wt\sB(G,F)$.
	\item We let $r_1= \frac{1}{2}$.
	\item We define $\phi_1:S(F) \ra \bC^\times$ by \(
	\phi_1\left( \begin{pmatrix}
	a & b \\  \varpi b & a 
	\end{pmatrix}  \right)
	= \varphi(2ab).
	\)
	\item $S_{[x]}=S(F)=\{\pm \Id\} \times S_{x,0+}$ and we let $\rho$ be the trivial representation on a one-dimensional vector space.	
\end{enumerate}
The supercuspidal representation constructed from this Yu datum following the recipe in the next subsection turns out to be the representation described in Section \ref{Section-simple-sc-example}.

\subsection{The construction of supercuspidal representations à la Yu} \label{Section-Yus-construction}
In this section we outline how Yu (\cite{Yu}) constructs from a Yu datum
$$((G_i)_{1 \leq i \leq n+1}, x, (r_i)_{1 \leq i \leq n}, \rho, (\phi_i)_{1 \leq i \leq n})$$
 a compact-mod-center open subgroup $\wt K$ and a representation $\wt \rho$ of $\wt K$ such that $\cind_{\wt K}^{G(F)}\wt\rho$ is an irreducible supercuspidal representation of $G(F)$.

The compact-mod-center open subgroup $\wt K$ is given by \index{notation}{Ktilde@$\wt K$}
$$ \wt K=(G_1)_{x,\frac{r_1}{2}}(G_2)_{x,\frac{r_2}{2}}\hdots (G_n)_{x,\frac{r_n}{2}}(G_{n+1})_{[x]} ,$$ 
where $(G_{n+1})_{[x]}$ \index{notation}{Gnx@$(G_{n+1})_{[x]}$} denotes the stabilizer in $G_{n+1}(F)$ of the point $[x]$ in the (reduced) Bruhat--Tits building $\sB(G_{n+1}, F)$.

The representation $\wt \rho$ \index{notation}{rho@$\wt \rho$} is a tensor product of two representations $\rho$ and $\kappa$,
$$ \wt \rho = \rho \otimes \kappa ,$$
 where $\rho$ also denotes the extension of the representation $\rho$ of  $(G_{n+1})_{[x]}$ to $\wt K$ that is trivial on $(G_1)_{x,\frac{r_1}{2}}(G_2)_{x,\frac{r_2}{2}}\hdots (G_n)_{x,\frac{r_n}{2}}$. The representation $\kappa$ is built out of the characters $\phi_1, \hdots , \phi_n$. If $n=0$, then $\kappa$ is trivial and we are in the setting of depth-zero representations. \index{notation}{kappa@$\kappa$}
 
 We will first sketch the construction of $\kappa$ in the case $n=1$, i.e., when the Yu datum is of the form $((G=G_1 \supset G_2=G_{n+1}), x, (r_1), \rho, (\phi_1))$. 
  To simplify notation, we write $r=r_1$ and $\phi=\phi_1$, and we assume $G_1 \neq G_2$. In this case $\wt K=(G_1)_{x,\frac{r}{2}}(G_2)_{[x]}$.

\textbf{Step 1 (extending the character $\phi$ as far as possible).} The first step consists of extending the character $\phi$ to a character $\hat \phi$ of $(G_1)_{x,\frac{r}{2}+}(G_2)_{[x]}$. This is done as in Section \ref{Section-more-positive-depth-sc} by sending the root groups outside $G_2$ to 1. More precisely, $\hat \phi$ is the unique character of $(G_1)_{x,\frac{r}{2}+}(G_2)_{[x]}$ that satisfies \index{notation}{phihat@$\hat \phi$}
\begin{itemize}
	\item $\hat \phi|_{(G_2)_{[x]}}=\phi$, and
	\item $\hat \phi|_{(G_1)_{x,\frac{r}{2}+}}$ factors through 
	\begin{eqnarray*} 
		(G_1)_{x,\frac{r}{2}+}/(G_1)_{x,r+} & \simeq& \fg_{x,\frac{r}{2}+}/\fg_{x,r+} = (\fg_2(F) \oplus \fr(F))_{x,\frac{r}{2}+}/(\fg_{2}(F) \oplus \fr(F))_{x,r+} \\
		& \ra & (\fg_{2})_{x,\frac{r}{2}+}/(\fg_{2})_{x,r+} \simeq 
		(G_2)_{x,\frac{r}{2}+}/(G_2)_{x,r+},
	\end{eqnarray*}
	on which it is induced by $\phi|_{(G_2)_{x,\frac{r}{2}+}}$, where we use the Moy--Prasad isomorphism, $\fr(F)$ is defined to be 
	$$\fr(F)= \fg(F) \cap \bigoplus_{\alpha \in \Phi(G_E,T_E) \smallsetminus \Phi((G_2)_E,T_E)} \fg(E)_\alpha$$
	for some maximal torus $T$ of $G_2$ that splits over a tamely ramified extension $E$ of $F$ with $x \in \wt\sA(T_E,E)$, and the surjection $\fg_2(F) \oplus \fr(F) \twoheadrightarrow \fg_2(F)$ sends $\fr(F)$ to zero. 
\end{itemize}

\textbf{Step 2 (Heisenberg representation).} As second step we extend the (one-dimensional) representation $\hat \phi|_{(G_1)_{x,\frac{r}{2}+}(G_2)_{x,\frac{r}{2}}}$ to a representation $(\omega, V_\omega)$ of $(G_1)_{x,\frac{r}{2}}$. We write $V_{\frac{r}{2}}$ for the quotient 
$$V_{\frac{r}{2}} =(G_1)_{x,\frac{r}{2}}/((G_1)_{x,\frac{r}{2}+}(G_2)_{x,\frac{r}{2}})$$
and we view  $V_{\frac{r}{2}}$ as an $\bF_p$-vector space. (It can also be viewed as an $\bF_q$-vector space, but here we only consider the underlying $\bF_p$-vector space structure.) Then one can show that the pairing 
$$ \<g, h \>:= \hat\phi(ghg^{-1}h^{-1}), \, \, g, h \in (G_1)_{x,\frac{r}{2}}  $$
defines a non-degenerate symplectic form on $V_{\frac{r}{2}}=(G_1)_{x,\frac{r}{2}}/((G_1)_{x,\frac{r}{2}+}(G_2)_{x,\frac{r}{2}})$ when we choose an identification between the $p$-th roots of unity in $\bC^\times$ and $\bF_p$. 

Now the theory of Heisenberg representations implies that there exists a unique irreducible representation $(\omega, V_\omega)$ of $(G_1)_{x,\frac{r}{2}}$ that restricted to $(G_1)_{x,\frac{r}{2}+}(G_2)_{x,\frac{r}{2}}$ acts via $\hat\phi$ (times identity), and the dimension of $V_\omega$ is $\sqrt{\#V_{\frac{r}{2}}}=p^{(\dim_{\bF_p}V_{\frac{r}{2}})/2}$.

\textbf{Step 3 (Weil representation).} The final step of the construction consists of extending the action of $(G_1)_{x,\frac{r}{2}}$ on $V_\omega$ via $\omega$ to an action of $\wt K=(G_1)_{x,\frac{r}{2}}(G_2)_{[x]}$ on $V_\omega$ by defining an action of $(G_2)_{[x]}$ on $V_\omega$ that is compatible with $\omega$. In order to obtain this action, we first observe that $(G_2)_{[x]}$ acts on $V_{\frac{r}{2}}$ via conjugation and that this action preserves the symplectic form $\< \cdot , \cdot\>$. This provides a morphism from $(G_2)_{[x]}$ to the group $\Sp(V_{\frac{r}{2}})$ of symplectic isomorphisms of $V_{\frac{r}{2}}$. Now the Weil representation is a representation of the symplectic group $\Sp(V_{\frac{r}{2}})$ on the space $V_\omega$ of the Heisenberg representation of the symplectic vector space that is compatible with the Heisenberg representation in the following sense. Using the composition of the morphism $(G_2)_{[x]} \ra \Sp(V_{\frac{r}{2}})$ with the Weil representation tensored with the character $\phi$ allows us to extend the representation $(\omega, V_\omega)$ from $(G_1)_{x,\frac{r}{2}}$  to $(G_1)_{x,\frac{r}{2}}(G_2)_{[x]}$. We denote the resulting representation of $\wt K= (G_1)_{x,\frac{r}{2}}(G_2)_{[x]}$ also by $(\omega, V_\omega)$ and set $(\kappa, V_\kappa)=(\omega, V_\omega)$.

This concludes the construction of $\kappa$ and hence $\wt \rho=\rho \otimes \kappa$ in the case of $n=1$. For a more general Yu datum  $((G_i)_{1 \leq i \leq n+1}, x, (r_i)_{1 \leq i \leq n}, \rho, (\phi_i)_{1 \leq i \leq n})$ with $n>1$ we construct from each character $\phi_i$ $(1 \leq i \leq n)$ a representation $(\omega_i, V_{\omega_i})$ analogous to the construction of $(\omega,V_\omega)$ above. Then we define $\kappa$ to be the tensor product of all those representations, i.e.
$$ (\kappa, V_\kappa)=\left(\bigotimes_{1 \leq i \leq n}\omega_i, \bigotimes_{1 \leq i \leq n} V_{\omega_i}\right) .$$
For the details we refer the reader to \cite[\S2.5]{Fi-Yu-works}, which is based on \cite{Yu}.

\begin{Thm}[\cite{Yu, Fi-Yu-works}] \label{Thm-sc}
	The representation $\cind_{\wt K}^{G(F)}\wt\rho$ is a supercuspidal smooth irreducible representation of $G(F)$.
\end{Thm}  
We will sketch the structure of the proof in the next subsection.

\subsection{Sketch of the proof that the representations are supercuspidal}

In order to prove that $\cind_{\wt K}^{G(F)} \wt \rho$ is supercuspidal it suffices to prove that it is irreducible by Lemma \ref{Lemma-sc}. First one notes that $\wt \rho$ itself is irreducible. We assume that an element $g \in G(F)$ intertwines $\wt \rho$. Now the main task is to show that $g \in \wt K$ so that we can apply Lemma \ref{Lemma-intertwiner}. This is done in two steps.

\textbf{Step 1.} We show recursively that $g \in \wt K G_{n+1} \wt K$ using that the characters $\phi_i$ are generic.

The key part for this step is \cite[Theorem~9.4]{Yu}, which in the example of $n=1$ spelled out above implies the following lemma.
\begin{Lemma}[\cite{Yu}]
	Suppose that $g$ intertwines $\hat \phi|_{(G_1)_{x,\frac{r}{2}+}}$. Then $g\in (G_1)_{x,\frac{r}{2}}G_2(F)(G_1)_{x,\frac{r}{2}}$.
\end{Lemma}
As mentioned above, this lemma crucially uses the fact that $\phi$ is $(G, G_2)$-generic relative to $x$ of depth $r$ (if $G_1 \neq G_2$) and we refer to \cite[Theorem~9.4]{Yu} for the proof.

\textbf{Step 2.} By Step 1 we may assume that $g \in G_{n+1}(F)$. Step 2 consists of showing that then $g \in (G_{n+1})_{[x]}$ using the structure of the Weil--Heisenberg representation and that $\rho|_{(G_{n+1})_{x,0}}$ is cuspidal. The spirit of this step is similar to the proof of Theorem \ref{Thm-depth-zero}, but in this more general setting it additionally requires an intricate study of the involved Weil--Heisenberg representations.

The reader interested in the full details of the proof is encouraged to read \cite[\S3]{Fi-Yu-works}, which is only about four pages long and refers to precise statements in \cite{Yu} that allow an easy backtracking within \cite{Yu} if the reader is interested in all the details that make the complete proof.

\section{Classification of supercuspidal representations} \label{Section-classification}
We keep the notation from the previous section including the assumption that $p$ is odd. 

In Section \ref{Section-Yus-construction} we outlined how to attach a supercuspidal representation to a Yu datum, which was described in Section \ref{Section-input}. In Section \ref{Subsection-exhaustion} we will see that under mild assumptions this provides us with all supercuspidal irreducible smooth representations. In order to parameterize all supercuspidal irreducible representations it therefore remains to understand which Yu data yield the same representation. This has been resolved by Hakim and Murnaghan (\cite{HakimMurnaghan}) up to a hypothesis that was removed  by Kaletha (\cite[\S~3.5]{Kaletha-regular}) and is sketched in Section \ref{Section-HM}. In Section \ref{Section-epsilon} we discuss the suggestion of Fintzen, Kaletha and Spice (\cite{FKS}) to twist the resulting parametrization of supercuspidal representations by a quadratic character and indicate some of its advantages. Section \ref{Section-regular-supercuspidal} explains how to reinterpret a large class of the supercuspidal representations that Yu constructed in terms of much simpler data consisting only of an appropriate pair of a torus $S$ and a character $\theta$, following Kaletha (\cite{Kaletha-regular}). This generalizes the representations that we constructed in Section \ref{Section-more-positive-depth-sc} and are the representations that we alluded to in Section \ref{Section-connection-real-p-adic} for which we know the Harish-Chandra character under some assumptions on $F$, see Section \ref{Section-characters}, and have a local Langlands correspondence (\cites{Kaletha-regular, FKS}).

\subsection{A parametrization of supercuspidal representations} \label{Section-HM}
 Hakim and Murnaghan define an equivalence relation on the Yu data, which they call \textit{$G(F)$-equivalence} and the key result is that two supercuspidal representations arising from Yu's construction are equivalent if and only if the input Yu data are $G(F)$-equivalent. In order to define the $G(F)$-equivalence, Hakim and Murnaghan introduced the following three transformations of Yu data.  

\begin{Def}[Elementary transformation]
	A Yu datum $\left((G_i)_{1 \leq i \leq n+1}, x', (r_i)_{1 \leq i \leq n}, \rho',\right.$ $\left.(\phi_i)_{1 \leq i \leq n}\right)$ is obtained from a Yu datum $((G_i)_{1 \leq i \leq n+1}, x, (r_i)_{1 \leq i \leq n}, \rho, (\phi_i)_{1 \leq i \leq n})$ via an \textit{elementary transformation} if $[x]=[x']$ and $\rho \simeq \rho'$. \index{definition}{elementary transformation}
\end{Def}

\begin{Def}[$G$-conjugation]
	We say that a Yu datum 
	is a obtained from the Yu datum $((G_i)_{1 \leq i \leq n+1}, x, (r_i)_{1 \leq i \leq n}, \rho, (\phi_i)_{1 \leq i \leq n})$ via \textit{$G(F)$-conjugation} if it is of the form 	$$((gG_ig^{-1})_{1 \leq i \leq n+1}, g.x, (r_i)_{1 \leq i \leq n}, {{^g}\rho}, ({{^g}\phi_i})_{1 \leq i \leq n})$$
	for some $g \in G(F)$.
\end{Def}

While the above two operations clearly yield isomorphic representations, there is a third operation on the Yu datum that does not change the isomorphism class of the resulting supercuspidal representation.

\begin{Def}[Refactorization]
	A Yu datum $((G_i)_{1 \leq i \leq n+1}, x, (r_i)_{1 \leq i \leq n}, \rho', (\phi'_i)_{1 \leq i \leq n})$ is a \textit{refactorization} \index{definition}{refactorization} of a Yu datum $((G_i)_{1 \leq i \leq n+1}, x, (r_i)_{1 \leq i \leq n}, \rho, (\phi_i)_{1 \leq i \leq n})$ if the following two conditions are satisfied.
	\begin{enumerate}[(i)]
		\item For $1 \leq i \leq n $, we have
		$$\prod_{1 \leq j \leq i} \phi_j|_{(G_{i+1})_{x,r_{i+1}+}}=\prod_{1 \leq j \leq i} \phi'_j|_{(G_{i+1})_{x,r_{i+1}+}} ,$$
		where we set $r_{n+1}=0$, and
		\item $$\rho \otimes \prod_{1 \leq j \leq n} \phi_j|_{(G_{n+1})_{[x]}}=\rho' \otimes \prod_{1 \leq j \leq n} \phi'_j|_{(G_{n+1})_{[x]}} .$$
	\end{enumerate}
\end{Def}

These three operations together allow us to define the desired equivalence relation on the Yu data.

\begin{Def}
	Two Yu data are \textit{$G(F)$-equivalent} if one can be transformed into the other via a finite sequence of refactorizations, $G(F)$-conjugations and elementary transformations.
\end{Def}

The following theorem shows that this is the equivalence relation we were looking for.

\begin{Thm}[\cites{HakimMurnaghan, Kaletha-regular}] \label{Thm-equivalent-Yu-data}
	Two Yu data $((G_i)_{1 \leq i \leq n+1}, x, (r_i)_{1 \leq i \leq n}, \rho, (\phi_i)_{1 \leq i \leq n})$ and $((G'_i)_{1 \leq i \leq n+1}, x', (r'_i)_{1 \leq i \leq n}, \rho', (\phi'_i)_{1 \leq i \leq n})$ yield isomorphic supercuspidal representations of $G(F)$ if and only if they are $G(F)$-equivalent.
\end{Thm}

For a proof, see \cite[Theorem~6.6]{HakimMurnaghan} and \cite[Corollary~3.5.5.]{Kaletha-regular}.

\subsection{A twist of Yu's construction} \label{Section-epsilon}
Let  $((G_i)_{1 \leq i \leq n+1}, x, (r_i)_{1 \leq i \leq n}, \rho, (\phi_i)_{1 \leq i \leq n})$ be a Yu datum. Instead of associating to this Yu datum the representation $\cind_{\wt K}^{G(F)}\wt \rho$ constructed by Yu, a new suggestion by Fintzen, Kaletha and Spice (\cite{FKS}) consists of associating the representation $\cind_{\wt K}^{G(F)}(\eps \wt \rho)$ for an explicitly constructed character $\eps: \wt K \ra \{\pm 1\}$. We refer the reader to \cite[p.~2259]{FKS} for the definition of $\eps$ as it is rather involved. 
 There are multiple reasons for the introduction of this quadratic twist in the parametrization. For example, it restores the validity of Yu's original proof (\cite{Yu}) that $\cind_{\wt K}^{G(F)}(\eps \wt \rho)$ is a supercuspidal irreducible representation, which is not valid for the non-twisted version as it relied on a misprinted statement in \cite{Gerardin}. In particular, we restore the validity of the intertwining results \cite[Proposition~14.1 and Theorem~14.2]{Yu} for the twisted construction that form the heart of Yu's proof. Instead of stating the results in full generality, which would involve introducing additional notation, we state its implication in the setting that we already introduced above. 
\begin{Prop}[\cites{Yu, FKS}]  Let  $((G=G_1 \supsetneq G_2=G_{n+1}), x, (r_1=r), \rho, (\phi_1=\phi))$ be a Yu datum from which we construct a representation $\kappa$ of $\wt K=(G_1)_{x,\frac{r}{2}}(G_2)_{[x]}$ as in Section \ref{Section-Yus-construction}. Then for $g \in G_2(F)$, we have
	$$\dim_{\bC} \Hom_{\wt K \cap g\wt Kg^{-1}}(\eps \kappa, {^g(\eps \kappa)})=1 .$$
\end{Prop}	
This result also holds in a more general setting in which we drop the assumption that $Z(G_2)/Z(G)$ is anisotropic.
We refer the reader to \cite[Corollary~4.1.11 and Corollary~4.1.12]{FKS} for the detailed statements and proofs.

Applications of the existence of the above quadratic character $\eps: \wt K \ra \{\pm 1\}$ include being able (under some assumptions on $F$) to provide a character formula for the supercuspidal representations $\cind_{\wt K}^{G(F)}(\eps \wt \rho)$ (\cite{Spice18, Spice21v2, FKS}), to suggest a local Langlands correspondence for all supercuspidal Langlands parameters (\cite{Kaletha-non-singular}) and to prove the stability and many instances of the endoscopic character identities for the resulting supercuspidal L-packets that such a local Langlands correspondence is predicted to satisfy (\cite{FKS}).


\subsection{Exhaustiveness of the construction of supercuspidal representations} 
\label{Subsection-exhaustion}

\begin{Thm}[\cites{Kim, Fi-exhaustion}] \label{Thm-exhaustion}
	Suppose that $G$ splits over a tamely ramified field extension of $F$ and that $p$ does not divide the order of the absolute Weyl group of $G$. Then every supercuspidal smooth irreducible representation of $G(F)$ arises from Yu's construction, i.e., via Theorem \ref{Thm-sc}.
\end{Thm}	

This result was shown by Kim (\cite{Kim}) in 2007 under the additional assumptions that $F$ has characteristic zero and that $p$ is ``very large''. Her approach was very different from the recent approach in \cite{Fi-exhaustion}. Kim proves statements about a measure one subset of all smooth irreducible representations of $G(F)$ by matching summands
of the Plancherel formula for the group and the Lie algebra, while the recent approach in \cite{Fi-exhaustion} is more explicit and can be used to recursively exhibit a Yu datum for the construction of the given representation. The latter approach consists of two main steps. The first step is to prove that every supercuspidal smooth irreducible representation of $G(F)$ contains a (maximal) datum as defined in \cite{Fi-exhaustion}, which can be viewed as a skeleton of a Yu datum. The second step consists of obtaining a Yu datum from that maximal datum and showing that the representation we started with is isomorphic to the one constructed from this Yu datum. We refer the reader to \cite{Fi-exhaustion} for the details and to Section 5 of \cite{Fi-CDM} for an expanded overview.

\subsection{Regular supercuspidal representations} \label{Section-regular-supercuspidal}
In this section we will reinterpret the Yu datum for a large class of supercuspidal representations in terms of a much simpler datum consisting only of an elliptic torus and an appropriate character thereof following Kaletha (\cite{Kaletha-regular}). This is a vast generalization of the representations we constructed in Section \ref{Section-more-positive-depth-sc}, which were attached to a pair of an elliptic torus $S$ and a $(G, S)$-generic character of $S(F)$, see Remark \ref{Rem-input-for-reps}. From now on, i.e., throughout this subsection and in Section \ref{Section-characters}, we assume that $p$ is odd, is not a bad prime for $G$, and does not divide the order of the fundamental group and the order of the center of the derived group of $G$. These conditions are satisfied if $p$ does not divide the order of the absolute Weyl group of $G$, for example. 

 The input for the construction in this subsection is a regular tame elliptic pair $(S, \theta)$, which is defined as follows.
\begin{Def}[{\cite[Definition~3.7.5]{Kaletha-regular}}] \label{Def-tame-regular-pair}
	A pair $(S, \theta)$ consisting of a maximal torus $S \subset G$ and a character $\theta: S(F) \rightarrow \bC^\times$  is called a \textit{regular tame elliptic pair} if 
	\begin{enumerate}[(1)] 
		\item $S \subset G$ is elliptic and splits over a tamely ramified extension $E$ over $F$,
		\item the action of the inertia subgroup of the absolute Galois group of $F$ on the root system 
		\[\Phi_{0+}:=\{ \alpha \in \Phi(G,S) \, | \, \theta(N_{E/F}(\check\alpha(1+\varpi_E\cO_E)))=1 \}
		 \]
		preserves a subset of positive roots $\Phi_{0+}^+ \subset \Phi_{0+}$, where $N_{E/F}$ denotes the norm map for the field extension $E/F$,
		\item the stabilizer in $N_{G^0}(S)(F)/S(F)$ of the restriction $\theta|_{S(F)_0}$ of $\theta$ to $S(F)_0$ is trivial, where $G^0 \subset G$ denotes the connected reductive subgroup of $G$ with maximal torus $S$ and root system $\Phi_{0+}$.
	\end{enumerate}
\end{Def}
Here $S(F)_0$ denotes $S(F)_{x,0}$ for $x$ the unique point of $\sB(S,F)$.

To each regular tame elliptic pair $(S, \theta)$, we will attach a supercuspidal representation $\pi_{(S,\theta)}$. In Section \ref{Section-characters} we will provide a formula for the Harish-Chandra character of $\pi_{(S,\theta)}$ under some assumptions on $F$. The construction of $\pi_{(S,\theta)}$ is achieved by using Yu's construction and twisting it by the quadratic character $\eps$ of Section \ref{Section-epsilon}. To do so,  we fix a regular tame elliptic pair $(S, \theta)$ and construct a Yu datum from it following Kaletha (\cite[\S3.6 and 3.7]{Kaletha-regular}):
\begin{enumerate}[(a)]
	\item In order to construct a twisted Levi sequence $G=G_1 \supseteq G_2 \supsetneq G_3 \supsetneq \hdots \supsetneq G_{n+1}$, for each $r \in \bR_{>0}$ we define the subset $\Phi_r \subset \Phi(G,S)$ by
	\[\Phi_{r}=\{ \alpha \in \Phi(G,S) \, | \, \theta(N_{E/F}(\check\alpha(1+\varpi_E^{\ceil{er}}\cO_E))=1 \} ,
	\] 
	where $e$ is the ramification index of $E/F$, and we write $\Phi_{r+} = \bigcap_{s>r} \Phi_s$. These subsets are stable under the Galois action. The set $\{ r \in \bR_{>0} \, | \, \Phi_{r} \neq \Phi_{r+} \}$ is finite and we let $r_1 > r_2 > \hdots > r_n >0$ be the real numbers that satisfy
	\begin{equation} \label{eqn-jumps}
		 \{ r \in \bR_{>0} \, | \, \Phi_{r} \neq \Phi_{r+} \} \cup \{\text{depth of } \phi\}= \{ r_1, r_2, \hdots, r_n \} 
	\end{equation}
	 We set $r_{n+1}=0$ and define $G_i$ to be the connected reductive subgroup of $G$ with maximal torus $S$ and root system $\Phi_{r_i+}$ for $1 \leq i \leq n+1$. By \cite[Lemma~3.6.1]{Kaletha-regular} and the assumption that $S$ splits over a tamely ramified extension $E/F$, the sequence $G=G_1 \supseteq G_2 \supsetneq G_3 \supsetneq \hdots \supsetneq G_{n+1}$ consists of twisted Levi subgroups that split over the tamely ramified extension $E/F$.
	\item  For $x$ we choose any point in $\wt\sB(S,F)\subset \wt\sB(G_{n+1},F)\subset \wt\sB(G,F)$.
	\item $r_1 > r_2 > \hdots > r_n >0$ are the real numbers satisfying Equation \eqref{eqn-jumps}.  
	\item $\phi_i$, for $1 \leq i \leq n$, is a character of $G_{i+1}(F)$ of depth $r_i$ that is trivial on the image of the $F$-points $(G_{i+1})_{\mathrm{sc}}(F)$ of the simply connected cover of the derived subgroup of $G_{i+1}$ and that is $(G_i, G_{i+1})$-generic relative to $x$ of depth $r_i$ if $G_i \neq G_{i+1}$, such that 
	\[\theta=\prod_{i=1}^{n+1} \phi_i|_{S(F)} ,\]
	where $\phi_{n+1}$ is trivial if $G_{n+1}=S$ and otherwise is a depth-zero character of $S(F)$.
	Such characters exist by \cite[Proposition~3.6.7]{Kaletha-regular} and are called a \textit{Howe factorization}\index{definition}{Howe factorization} of $(S, \theta)$ (\cite[Definition~3.6.2]{Kaletha-regular}).
	\item $\rho$ is an irreducible representation of $(G_{n+1})_{[x]}$ constructed in \cite[\S3.4.4]{Kaletha-regular} with the property that it is trivial on $(G_{n+1})_{x,0+}$ and that its restriction to $(G_{n+1})_{x, 0}$ is isomorphic to the inflation of the irreducible cuspidal Deligne--Lusztig representation $\pm R_{\bar S, \bar \phi_{n+1}}$ (\cite{DeligneLusztig}). Here $\bar S$  is the reductive quotient of the special fiber of the connected Néron model of $S$, which means that under the identification of $(\bG_{n+1})_x(\bF_q)$ with $(G_{n+1})_{x,0}/(G_{n+1})_{x,0+}$ the $\bF_q$-points $\bar S(\bF_q)$ of $\bar S$ are identified with $S_{x,0}/S_{x,0+}$. The character $\bar \phi_{n+1}$ is the restriction of $\phi_{n+1}$ to $S_{x,0}$ that factors through $S_{x,0}/S_{x,0+}$ and hence yields a character of $S_{x,0}/S_{x,0+}$. The assumption that $(S, \theta)$ is a regular tame elliptic pair ensures by \cite[Fact~3.4.18 and Lemma~3.6.5]{Kaletha-regular} that $\bar \phi_{n+1}$ is in general position in the notion of Deligne and Lusztig (\cite[Definition~5.15]{DeligneLusztig}) so that $\pm R_{\bar S, \bar \phi_{n+1}}$ is an irreducible cuspidal representation of $(G_{n+1})_{x,0}/(G_{n+1})_{x,0+}$ (\cite[Proposition~7.4 and Theorem~8.3]{DeligneLusztig}).	
\end{enumerate}

From this Yu datum attached to $(S, \theta)$, Yu's construction sketched in Section \ref{Section-Yus-construction} provides a compact open subgroup $\wt K$ with a representation $\wt \rho$, and we denote by $\pi_{(S, \theta)}$ \index{notation}{piStheta@$\pi_{(S, \theta)}$} the resulting supercuspidal representation  $\cind_{\wt K}^{G(F)}(\eps \wt \rho)$ obtained by twisting $\wt \rho$ by the quadratic character $\eps$ from Section \ref{Section-epsilon} before compactly inducing it from $\wt K$ to $G(F)$.

While the Yu datum constructed above relied on the choice of $x$ and a Howe factorization 	$\theta=\prod_{i=1}^{n+1} \phi_i|_{S(F)}$ of $\theta$, the resulting representation  $\pi_{(S, \theta)}$ is independent of these choices. This follows from Theorem \ref{Thm-equivalent-Yu-data}, because the resulting Yu data can be transformed into each other via an elementary transformation and a refactorization by \cite[Lemma~3.6.6]{Kaletha-regular}. Hence the following fact follows from 
\cite[Proposition~3.7.8]{Kaletha-regular}.
	
\begin{Fact}[{\cite{Kaletha-regular}}]  
	The supercuspidal representation $\pi_{(S, \theta)}$ is well defined. If $(S', \theta')$ is another regular tame elliptic pair, then $\pi_{(S, \theta)} \simeq \pi_{(S', \theta')}$ if and only if  $(S', \theta')$ is a $G(F)$-conjugate of $(S, \theta)$.
\end{Fact}

\begin{Not}
	We call the representations $\pi_{(S, \theta)}$ arising from regular tame elliptic pairs $(S, \theta)$ \textit{regular supercuspidal representations}.
\end{Not}

Thus, the regular supercuspidal representations are parameterized by $G(F)$-conjugacy classes of regular tame elliptic pairs. 

\begin{Rem}
	We note that our parametrization here differs from Kaletha's (\cite[Corollary~3.7.10]{Kaletha-regular}) by the inclusion of the quadratic character $\epsilon$ that was not yet available at the time of writing of Kaletha's paper \cite{Kaletha-regular}. This version seems to fit a bit better into the Langlands program. More precisely, using Kaletha's initial parametrization required him to twist by an auxiliary quadratic character in his construction of the local Langlands correspondence. With the parametrization presented here that includes the twist by the quadratic character $\eps$ of \cite{FKS}, this twist happens already as part of the parametrization $(S, \theta) \mapsto \pi_{{(S, \theta)}}$ and hence no additional twist is necessary in the construction of the local Langlands correspondence.
	
	On the other hand, Kaletha (\cite{Kaletha-doublecover-tori}) suggests that from the Langlands parameter side it would seem more natural to attach a supercuspidal representation to a  genuine character of a double cover of the torus $S$, which would also simplify the character formulas discussed in Section \ref{Section-characters} below. Such a phenomenon is also observed in the setting of real Lie groups, where representations can be attached to characters of a double cover of the torus $S$, cf.\ \cite{Adams-Kaletha} which is built on \cite{HC-DS1}. Such a construction is still missing in the setting of $p$-adic groups.
\end{Rem}


\section{Harish-Chandra characters of supercuspidal representations} \label{Section-characters}
We recall from \cite[\S3.4]{Taibi-IHES}(which will appear in the same proceedings as this article) that the Harish-Chandra character distribution attached to an irreducible smooth representation $\pi$ can be represented by a unique locally constant function $\Theta_\pi$ on the subset of regular semisimple elements $G_{\mathrm{rs}}(F)$ of $G(F)$ (\cite[Theorem~16.3]{HC99}), which we call the Harish-Chandra character, and that the Harish-Chandra character determines the irreducible representation uniquely up to isomorphism. While the above references only treat the case where $F$ has characteristic zero, the same results hold for fields $F$ of positive characteristic, see for example \cite[\S13]{Adler-Korman}, which is based on work of Gopal Prasad and Harish-Chandra.

We keep the notation and assumptions from Section \ref{Section-regular-supercuspidal}, which includes fixing a regular tame elliptic pair $(S, \theta)$ and denoting by $E$ the splitting field of $S$. In this section we  present (under some assumption) a formula for the Harish-Chandra character of the representation $\pi_{(S, \theta)}$ constructed in Section \ref{Section-regular-supercuspidal}. While the results stated here appeared first in \cite{FKS}, they are based on prior work of Adler--Spice (\cite{Adler-Spice-characters}), DeBacker--Reeder (\cite{DeBackerReeder}), DeBacker--Spice (\cite{DeBacker-Spice}), Kaletha (\cite{Kaletha-regular}) and Spice (\cites{Spice18, Spice21v2}).
Following \cite{FKS} we start with providing a character formula that only holds for specific semisimple elements as this version is easier to state and requires less assumptions.

\subsection{The character of regular supercuspidal representations at topologically semisimple modulo center elements}\label{Section-character-shallow}

The Harish-Chandra character $\Theta_{\pi_{(S,\theta)}}$ of the representation $\pi_{(S,\theta)}$ evaluated at appropriately nice elements $\gamma$ turns out to have the following shape
$$\Theta_{\pi_{(S,\theta)}}(\gamma)=(-1)^{??} \cdot  \sum_{g \in N_G(S)(F) / S(F)} \frac{\theta({g\gamma g^{-1}})}{?(g\gamma g^{-1})},$$
where we hope the reader recognizes already a similarity to the character formula of an essentially square-integrable representation of a real reductive group described in \eqref{eqn:char:real}.

In order to make the character formula more precise in Theorem \ref{Thm-character-formula-shallow} below, we need to explain what appropriately nice elements $\gamma$ we are considering and introduce some notation that will replace ? and ?? in the above equation by explicit expressions.

\begin{Def}
	Let $Z$ be a subgroup in the center $Z(G)$ of $G$.
	We call an element $\gamma \in G(F)$ \textit{topologically semisimple modulo $Z$}\index{definition}{topologically semisimple (modulo $Z$)} if $\gamma$ is  semisimple and for every maximal torus $T$ in the centralizer of $\gamma$ and every $\chi \in X^*(T_{\bar F}/Z_{\bar F})$ the element $\chi(\gamma)$ has finite order coprime to $p$.
	
	We call an element $\gamma \in G(F)$ \textit{topologically unipotent modulo $Z$}\index{definition}{topologically unipotent (modulo $Z$)} if the image $\bar \gamma$ of $\gamma$ in $G/Z$  satisfies $\lim_{n \ra \infty} \bar \gamma^{p^n}=1$.
	
	For $\gamma \in G(F)$, we say that $\gamma=\gamma_0 \gamma_{0+}$ is a \textit{topological Jordan decomposition modulo $Z$}\index{definition}{topological Jordan decomposition (modulo $Z$)} if  $\gamma =\gamma_0 \gamma_{0+}= \gamma_{0+} \gamma_{0}$, the element $\gamma_0 \in G(F)$ is topologically semisimple modulo $Z$ and the element  $\gamma_{0+} \in G(F)$ is topologically unipotent modulo $Z$.
\end{Def}
These properties were introduced in a more general frame work in \cite{Spice-Jordan-decomposition} and  called absolutely $F$-semisimple modulo $Z$, topologically $F$-unipotent modulo $Z$ (\cite[Lemma~2.21]{Spice-Jordan-decomposition}), and a topological $F$-Jordan decomposition modulo $Z$ (\cite[Proposition~2.42]{Spice-Jordan-decomposition}), respectively. The following result follows from \cite[Proposition~2.24 and Proposition~2.36]{Spice-Jordan-decomposition}.

\begin{Prop}[{\cite{Spice-Jordan-decomposition}}]
	If $\gamma \in G(F)$ is a compact-mod-center element, then $\gamma$ has a topological Jordan decomposition modulo $Z(G)$: $\gamma=\gamma_0\gamma_{0+}$. The elements $\gamma_0$ and $\gamma_{0+}$ are uniquely determined modulo $Z(G)(F)$.
\end{Prop}

By the following fact it suffices to describe the Harish-Chandra character of $\pi_{(S,\theta)}$ on compact-mod-center elements.

\begin{Fact}[{\cite[Théorème]{Deligne-support-caractere}}]
 The Harish-Chandra character of a supercuspidal representation is supported on the compact-mod-center elements.
\end{Fact}

In this subsection we will obtain a character formula for elements that are topologically semisimple modulo $Z(G)$, but we first need to introduce a bit more notation.

We denote by $e(G)$ the Kottwitz sign, which is defined as $(-1)^{r(G_{\mathrm{qs}})-r(G)}$, where $r(G)$ denotes the $F$-rank of the derived subgroup of $G$, and $G_{\mathrm{qs}}$ denotes the quasi-split inner form of $G$.

We let $T_G$ be a minimal Levi subgroup of the quasi-split inner form $G_{\mathrm{qs}}$ of $G$ and denote by 
$\eps_L(X^*(T_G)_\bC-X^*(S)_\bC, \varphi)$ the local $\epsilon$-factor as normalized by Langlands, see \cite[(3.6)]{Tate-NT}.

$D(\gamma)$ is the Weyl discriminant given by $D(\gamma)=\abs{\prod_{\alpha \in \Phi(G,S)}(1-\alpha(\gamma))}$ for $\gamma \in S(\bar F)$. Since the Weyl discriminant is invariant under the action of the normalizer of $S$, we may extend $D$ to all semisimple elements of $G(F)$ via $G(\bar F)$-conjugation.  

We choose a- and $\chi$-data consisting of $a_\alpha \in F_\alpha^\times$ and $\chi''_\alpha: F_\alpha^\times \ra \bC^\times$ for every $\alpha \in \Phi(G, S)$ as in \cite[\S4.2]{FKS} and refer the reader to  \textit{loc.\ cit} and \cite{Kaletha-regular} for precise definitions. Then we set $\Delta^{\mathrm{abs}}_{II}[a, \chi''](\gamma)=
\prod_{\alpha \in \Phi(G, S)/\Gal(\bar F/F) \atop -\alpha \in \Gal(\bar F/F)\alpha}\chi''_\alpha\left(\frac{\alpha(\gamma)-1}{a_\alpha}\right)$ for $\gamma \in S(F)$ and can now state the full character formula for topologically semisimple elements. 

\begin{Thm}[{\cite[Proposition~4.3.2]{FKS}}] \label{Thm-character-formula-shallow}
	Let $\gamma \in S(F) \subset G(F)$ be regular and topologically semisimple modulo $Z(G)$.
	Then 
	\begin{equation} \label{eqn:charformulashallow}
		\Theta_{\pi_{(S,\theta)}}(\gamma)=e(G)\eps_L(X^*(T_G)_\bC-X^*(S)_\bC, \varphi)  \sum_{g \in N(S)(F) / S(F)} D({^g\gamma})^{-\frac{1}{2}}\Delta^{\mathrm{abs}}_{II}[a, \chi'']({^g\gamma})\theta({^g\gamma}) .
	\end{equation}
\end{Thm}

If one interprets the terms in the character formula \eqref{eqn:charformulashallow} appropriately in the setting of reductive groups over $\bR$, then one recovers exactly the character formula \eqref{eqn:char:real} as was shown by Kaletha in \cite[\S4.11]{Kaletha-regular}. We refer the reader to \textit{loc.\ cit} for a detailed explanation.

\subsection{The character of regular supercuspidal representations in general}

We keep the notation from the previous subsection, and now we assume in addition that the characteristic of the local field $F$ is zero and that the exponential map $\exp$ for $G$ converges on $\cup_{x \in \sB(G, F)} \, \fg_{x,0+}$. This is satisfied if $p \geq (2+e)n$, where $e$ is the ramification index of $F/\bQ_p$ and $n$ is the dimension of the smallest faithful algebraic representation of $G$, for example. Then there exists an element $X \in \Lie^*(S)(F)$ that satisfies $\theta(\exp(Y))=\varphi(\<X,Y\>)$ for all $Y \in \fs_{x,0+}$. Using the decomposition $\fg(F)=\fs(F) \oplus \fr(F)$ with $\fr(F)= \fg(F) \cap \bigoplus_{\alpha \in \Phi(G,S)} \fg(E)_\alpha$ we can extend $X$ to an element of $\Lie^*(G)(F)$ by setting it to be trivial on $\fr(F)$. For $g\in G(F)$, we write $X^g=\Ad^*(g)^{-1}X$, i.e., $X^g \in \fg^*(F)$ satisfies $\<X^g, Y\> =\< X,{^gY}\>$ for all $Y \in \fg(F)$.

The Harish-Chandra character of $\pi_{(S,\theta)}$ evaluated at a regular semisimple element $\gamma \in G(F)$ with topological Jordan decomposition modulo $Z(G)$ given by $\gamma=\gamma_0\gamma_{0+}$ is a combination of the character formula for topological semisimple modulo $Z(G)$ elements and a contribution of $\gamma_{0+}$ via the inverse, denoted by $\log$, of the exponential map $\exp$ as argument for an appropriate Fourier transform of an orbital integral. To make this precise, we denote by $J$ the identity connected component $\Cent_G(\gamma_0)^\circ$ of the centralizer of $\gamma_0$ in $G$, and by $\hat\cO^{J}_{X^g}$  the renormalized function that represents the Fourier transform of the orbital integral through $X^g$ for the group $J$, see, e.g., \cite[\S2.2, page~8]{Spice21v2} for details.

\begin{Thm}[{\cite[Proposition~4.3.5]{FKS}}]
	Let $\gamma=\gamma_0\gamma_{0+}$ be a topological Jordan decomposition modulo $Z(G)$ of a compact-mod-center, regular semisimple element $\gamma \in G(F)$ whose centralizer splits over a tame extension of $F$. Then 
	\begin{eqnarray*}
		\Theta_{\pi_{(S,\theta)}}(\gamma)&=&e(G)e(J)\eps_L(X^*(T_G)_\bC-X^*(T_J)_\bC, \varphi) D(\gamma)^{-\frac{1}{2}}\\
	& & \cdot	 \sum_{g \in S(F) \backslash G(F) / J(F) \atop {^g\gamma_0} \in S(F)} \Delta^{\mathrm{abs}}_{II}[a, \chi'']({^g\gamma_0})\theta({^g\gamma_0})\hat\cO^{J}_{X^g}(\log \gamma_{0+})  .
	\end{eqnarray*}
\end{Thm}


\printindex{notation}{Selected notation}

\printindex{definition}{Selected definitions}

\bibliography{Fintzenbib}

\end{document}